\documentclass{cpcHACK}
 \usepackage{amsmath, amssymb} 
 \usepackage{mathrsfs}
 \usepackage{graphicx}
 \usepackage{psfrag}
 \usepackage{latexsym}
 \usepackage{amsfonts}
 \usepackage{url}

\allowdisplaybreaks

 \newtheorem{defi}[theorem]{Definition}
 \newtheorem{exm}[theorem]{Example}

 \newcommand{\e}{\epsilon}

 \newcommand{\C} {\mathbb C}

 \def \C{\mathbb{C}}
 
 \def \sign {\mbox{sign}}
 
 \def \l {\delta}
 \def \e {\epsilon}

  \def\qed{\vbox{\hrule \hbox{\vrule\hbox to
 5pt{\vbox to 6pt{\vfil}\hfil}\vrule}\hrule}}

\begin{document}

 \bibliographystyle{alpha}
 \unitlength=1cm

\title[Expressing Combinatorial Problems by Polynomial Equations]{Expressing Combinatorial Optimization Problems by Systems of Polynomial Equations and the Nullstellensatz}
    \author[J.A.~De~Loera, J.~Lee, S.~Margulies, S.~Onn]{
        \spreadout{J.A.~DE~LOERA}$^1$%
         \thanks{Research supported in part by an IBM Open Collaborative Research Award
         and by NSF grant DMS-0608785}\ ,\quad
        \spreadout{J. LEE}{$^2$},\quad
        \spreadout{S.~MARGULIES}{$^{3\dag}$} and
        \spreadout{S.~ONN}{$^4$%
        \thanks{ Research supported by the ISF (Israel Science Foundation) and by the fund for the promotion of research at Technion}
        }
        \\
        \affilskip
        {$^1$} Department of Mathematics, Univ. of California, Davis, California, USA\\
        (email: {\tt deloera@math.ucdavis.edu})\\
        \affilskip
        {$^2$} IBM T.J. Watson Research Center, Yorktown Heights, New York, USA\\
        (email: {\tt jonlee@us.ibm.com})\\
        \affilskip
        {$^3$} Department of Computer Science, Univ. of California, Davis, California, USA\\
        (email: {\tt smargulies@ucdavis.edu})\\
        \affilskip
        {$^4$}Davidson Faculty of IE \& M, Technion - Israel Institute of Technology, Haifa, Israel\\
        (email: {\tt onn@ie.technion.ac.il})
}

 \maketitle

 \begin{abstract}
 Systems of polynomial equations over the complex or real numbers
 can be used to model combinatorial problems.
 In this way, a combinatorial problem is feasible (e.g. a graph is 3-colorable,
 hamiltonian, etc.) if and only if a related system of polynomial equations has
 a solution. In the first part of this paper, we construct new
 polynomial encodings for the problems of finding in a graph its
 longest cycle, the largest planar subgraph, the edge-chromatic
 number, or the largest $k$-colorable subgraph.

 For an infeasible polynomial system, the (complex) Hilbert
 Nullstellensatz gives a certificate that the associated combinatorial
 problem is
 infeasible. Thus, unless $\text{P}=\text{NP}$~, there must exist
 an infinite sequence of infeasible instances of each hard combinatorial
problem for which the minimum degree of a Hilbert Nullstellensatz
certificate of the associated polynomial system grows.

We show that the minimum-degree of a Nullstellensatz
certificate for the non-existence of
a stable set of size
greater than the stability number of the graph is the stability
number of the graph. Moreover, such a certificate contains at least
one term per stable set of $G$~. In contrast, for non-3-colorability, we
found only graphs with Nullstellensatz certificates of degree four.

\end{abstract}

\section{Introduction}

 N. Alon \cite{alonsurvey} used the term {\em ``polynomial
 method''} to refer to the use of non-linear polynomials for solving
 combinatorial problems. Although the polynomial method is not yet as
 widely used by combinatorists as, for instance, polyhedral or
 probabilistic techniques, the literature in this subject continues to
 grow. Prior work on encoding combinatorial properties included colorings  \cite{alontarsi,deloera,eliahou,hillarwindfeldt,lovasz1,matiyasevich1,matiyasevich2,Mnuk},
stable sets  \cite{deloera,lili,lovasz1,villa1},
matchings  \cite{fischer},
and flows \cite{alontarsi,Mnuk,Onn}. Non-linear encodings of
 combinatorial problems are often compact. This contrasts
 with the exponential sizes of systems of linear inequalities that
 describe the convex hull of incidence vectors of many combinatorial
 structures (see \cite{yannakakis}). In this article we present new
 encodings for other combinatorial problems, and we discuss
 applications of polynomial encodings to combinatorial
 optimization and to computational complexity.

 Recent work demonstrates that one can derive good semidefinite programming
 relaxations for combinatorial optimization problems from the
 encodings of these problems as polynomial systems (see
\cite{moniquefranz} and references therein for details).
Lasserre \cite{lasserre2}, Laurent \cite{monique07} and Parrilo \cite{parrilo1,
   parrilo2} studied the problem of minimizing a general polynomial
 function $f(x)$ over an algebraic variety having only finitely many
 solutions. Laurent proved that when the variety consists of the
 solutions of a zero-dimensional radical ideal $I$~, there is a way to
 set up the optimization problem $\min \{f(x) : x \in \text{variety}(I)\}$ as a
 finite sequence of semidefinite programs terminating with the optimal
 solution (see \cite{monique07}).

 This immediately suggests an application of the polynomial method to
 combinatorial optimization problems: Encode your problem with
 polynomials equations in ${\mathbb R} [x_1,\dots,x_n]$ that generate
 a zero-dimensional (variety is finite) radical ideal, then generate
 the finite sequence of SDPs following the method in \cite{monique07}.
 This highlights the importance of finding systems of polynomials for
 various combinatorial optimization problems.  The first half of this paper
 proposes new polynomial system encodings for the problems, with
 respect to an input graph, of finding a longest cycle, a largest planar
 subgraph, a largest $k$-colorable subgraph, or a minimum edge
coloring. In particular, we establish the following result.

 \begin{theorem}
\label{encodingthm}

 \begin{enumerate}
 \renewcommand{\theenumi}{\arabic{enumi}.}

 \item A simple graph $G$ with nodes $1,\dots,n$ has a cycle of length $L$ if
   and only if the following zero-dimensional system of polynomial
   equations has a solution:

 \begin{equation} \label{sizecycle}
 \sum_{i=1}^n y_i=L~.
 \end{equation}

 For every node $i=1, \ldots ,n$:
 \begin{equation}\label{choosevertex}
 y_i(y_i-1)=0,  \quad \quad
\prod_{s = 1}^n(x_i - s) = 0~,
 \end{equation}

 \begin{equation} \label{cyclicity}
 y_i \prod_{j \in \hbox{Adj(i)}}(x_i -y_j x_j + y_j)(x_i - y_j x_j -y_j(L-1)) = 0~.
 \end{equation}

Here $Adj(i)$ denotes the set of nodes adjacent to node $i$~.

 \item Let $G$ be a simple graph with $n$ nodes and $m$ edges. $G$ has
   a planar subgraph with $K$ edges if and only if the following zero-dimensional system
   of equations has a solution:

 For every edge $\{i,j\} \in E(G)$:
 \begin{equation}\label{whichedges}
 z_{\{ij\}}^2-z_{\{ij\}}=0, \ \sum_{\{i,j\} \in E(G)} z_{\{ij\}}-K=0~.
 \end{equation}

 For $k=1,2,3$~, every node $i \in V(G)$ and every edge $\{i,j\} \in E(G)$:
\begin{equation} \label{linearext1}
 \prod_{s = 1}^{n+m}(x_{\{i\}k} - s) = 0, \qquad  \prod_{s = 1}^{n+m}(y_{\{ij\}k} - s) = 0~,
\end{equation}

\begin{equation}
s_k \left(\prod_{\stackrel{i,j \in V(G)}{i < j}} \big(x_{\{i\}k}-x_{\{j\}k}\big) \prod_{\stackrel{i \in V(G),}{\{u,v\} \in E(G)}} \big(x_{\{i\}k}-y_{\{uv\}k} \big) \prod_{\{i,j\}, \{u,v\} \in E(G)} \big(y_{\{ij\}k}-y_{\{uv\}k} \big)\right)=1~.
\end{equation}

 For $k=1,2,3$~, and for every pair of a node $i \in V(G)$ and incident edge $\{i,j\} \in E(G)$:
 \begin{equation}\label{linearext2}
 z_{\{ij\}}\big(y_{\{ij\}k}-x_{\{i\}k}-\Delta_{\{ij,i\}k} \big)=0~.
 \end{equation}
 \vskip10pt

For every pair of a node $i \in V(G)$ and edge $\{u,v\} \in E(G)$ that is not incident on $i$:
 \begin{equation} \label{nomore1}
z_{\{uv\}}\big(y_{\{uv\}1}-x_{\{i\}1}-\Delta_{\{uv,i\}1} \big)\big(y_{\{uv\}2}-x_{\{i\}2}-\Delta_{\{uv,i\}2} \big)\big(y_{\{uv\}3}-x_{\{i\}3}-\Delta_{\{uv,i\}3} \big)=0~,
 \end{equation}
 \begin{equation} \label{nomore1a}
z_{\{uv\}}\big(x_{\{i\}1} - y_{\{uv\}1} - \Delta_{\{i,uv\}1} \big)\big( x_{\{i\}2} - y_{\{uv\}2} -\Delta_{\{i,uv\}2} \big) \big( x_{\{i\}3} - y_{\{uv\}3}-\Delta_{\{i,uv\}3} \big)=0~.
 \end{equation}
  \vskip10pt

For every pair of edges $\{i,j\}, \{u,v\} \in E(G)$ (regardless of whether or not they share an endpoint):
\begin{equation} \label{nomore2}
z_{\{ij\}}z_{\{uv\}}\big(y_{\{ij\}1}-y_{\{uv\}1}-\Delta_{\{ij,uv\}1} \big)\big(y_{\{ij\}2}-y_{\{uv\}2}-\Delta_{\{ij,uv\}2} \big) \big(y_{\{ij\}3}-y_{\{uv\}3}-\Delta_{\{ij,uv\}3} \big)=0~,
 \end{equation}
\begin{equation} \label{nomore2a}
z_{\{ij\}}z_{\{uv\}}\big(y_{\{uv\}1}-y_{\{ij\}1}-\Delta_{\{uv,ij\}1} \big)\big(y_{\{uv\}2}-y_{\{ij\}2}-\Delta_{\{uv,ij\}2} \big) \big(y_{\{uv\}3}-y_{\{ij\}3}-\Delta_{\{uv,ij\}3} \big)=0~.
 \end{equation}
   \vskip10pt

For every pair of nodes $i,j \in V(G)$~, (regardless of whether or not they are adjacent):
\begin{equation}
\big(x_{\{i\}1}-x_{\{j\}1}- \Delta_{\{i,j\}1} \big) \big(x_{\{i\}2}-x_{\{j\}2}-\Delta_{\{i,j\}2}\big) \big(x_{\{i\}3}-x_{\{j\}3}- \Delta_{\{i,j\}3} \big)=0~,
\end{equation}
\begin{equation}
\big(x_{\{j\}1}-x_{\{i\}1}- \Delta_{\{j,i\}1} \big) \big(x_{\{j\}2}-x_{\{i\}2}-\Delta_{\{j,i\}2}\big) \big(x_{\{j\}3}-x_{\{i\}3}- \Delta_{\{j,i\}3} \big)=0~.
\end{equation}
  \vskip10pt

For every $\Delta_{\text{index}}$ (e.g., $\Delta_{\{ij,uv\}k},\Delta_{\{ij,i\}k}$~, etc.) variable appearing in the above system:
\begin{equation} \label{delta_constraint}
\prod_{d=1}^{n+m -1}\big(\Delta_{\text{index}} - d \big)=0~.
\end{equation}

\item A graph $G$ has a $k$-colorable subgraph with $R$ edges if and only if the following zero-dimensional system of equations has a solution:
\begin{equation}
\sum_{\{i,j\} \in E(G)} y_{ij} -R=0~.
\end{equation}

For every vertex $i \in V(G)$:
\begin{equation}
 x_i^k=1~.
\end{equation}

For every edge $\{i,j\} \in E(G)$:
\begin{equation}
y_{ij}^2-y_{ij}=0, \quad y_{ij}\big( x_i^{k-1}+x_i^{k-2}x_j+\dots+ x_j^{k-1}\big)=0~.
\end{equation}

\item Let $G$ be a simple graph with maximum vertex degree $\Delta$~. The graph $G$ has edge-chromatic number $\Delta$ if
and only if the following zero-dimensional system of polynomials has a solution:

For every edge $\{i,j\} \in E(G)$:
\begin{equation} \label{1_1_4_colors}
x_{ij}^{\Delta}= 1~.
\end{equation}

For every node $i \in V(G)$:
\begin{equation} \label{1_1_4_cyccity}
s_i\left(\prod_{\stackrel{j,k \in Adj(i)}{j < k}}(x_{ij} -x_{ik})\right) = 1~, \end{equation}
where $Adj(i)$ is the set of nodes adjacent to node $i$~.

\noindent[By Vizing's theorem, if the system has no solution, then $G$ has edge-chromatic number $\Delta+1$~.]

 \end{enumerate}
 \end{theorem}

 In the second half of the article, we look at the connection between
 polynomial systems and computational complexity. We have already mentioned that semidefinite programming is one way to approach
 optimization. It is natural to ask how big are such SDPs. For simplicity of analysis, we look at the case of feasibility
 instead of optimization. In this case, the SDPs are replaced by a large-scale linear algebra problem. We will discuss
 details in Section \ref{preliminariesNULL}. For a hard
 optimization problem, say Max-Cut, we associate a system of
 polynomial equations $J$ such that the system has a solution if and only if the problem has a feasible solution. On the
 other hand, the famous
 Hilbert Nullstellensatz (see \cite{coxetal}) states that a system of
 polynomial equations
 $J=\left\{f_1(x)=0,f_2(x)=0,\dots,f_r(x)=0\right\}$ with complex
 coefficients has no solution in ${\mathbb C}^n$ if and only if there
 exist polynomials $\alpha_1,\dots,\alpha_r \in {\mathbb
 C}[x_1,\dots,x_n]$ such that $1=\sum \alpha_if_i$~.  Thus, if the
 polynomial system $J$ has no solution, there exists a {\em
 certificate} that the combinatorial optimization problem is
 infeasible.

 There are well-known \emph{upper bounds} for the degrees of the
 coefficients $\alpha_i$ in the Hilbert Nullstellensatz certificate
 for \emph{general}
 systems of polynomials, and they turn out to be sharp (see
 \cite{kollar}). For instance, the following well-known
 example shows that the degree of $\alpha_1$ is at least $d^m$~:
 \begin{align*}
 f_1=x_1^d, f_2=x_1-x_2^d,\dots, f_{m-1}=x_{m-2}-x_{m-1}^d, f_m=1-x_{m-1}x_m^{d-1}~.
 \end{align*}
 But polynomial systems for combinatorial optimization are special. One
 question is how complicated are the degrees of Nullstellensatz
 certificates of infeasibility?  As we will see in Section
 \ref{preliminariesNULL}, unless $\text{P}=\text{NP}$~,
for every hard combinatorial problem, there must exist
an infinite sequence of infeasible instances for which the
minimum degree of a Nullstellensatz certificate, for
the associated system of polynomials, grows arbitrarily large. This was first observed by L. Lov\'asz who proposed the problem of finding explicit graphs in \cite{lovasz1}.
A main contribution of this article is to exhibit such
 growth of degree explicitly.  In the second part of the paper we
 discuss the growth of degree for the NP-complete problems stable set
and 3-colorability. We establish the following theorem:

 \begin{theorem} \label{degHN}

 \begin{enumerate}
 \renewcommand{\theenumi}{\arabic{enumi}.}

 \item Given a graph $G$~, let $\alpha(G)$ denote its stability number. A minimum-degree
 Nullstellensatz certificate for the non-existence
 of a stable set of size greater than $\alpha(G)$
has degree equal to $\alpha(G)$ and contains at least one term per
stable set in $G$~.

 \item
Every Nullstellensatz certificate for non-3-colorability of a graph
has degree at least four. Moreover, in the case of a graph containing
an odd-wheel or a clique as a subgraph,
a minimum-degree Nullstellensatz certificate for non-3-colorability 
has degree exactly four.
 \end{enumerate}

 \end{theorem}

The paper is organized as follows.  Our encoding results for longest
cycle and largest planar subgraph appear in Subsection
\ref{newencodings1}. As a direct consequence, we recover a polynomial
system characterization of the hamiltonian cycle problem.  Similarly,
we discuss how to express, in terms of polynomials, the decision
question of whether a poset has dimension $p$~. The encodings for
edge-chromatic number and largest $k$-colorable subgraph also appear in
Subsection \ref{newencodings1}. As we mentioned earlier, colorability
problems were among the first studied using the polynomial method; we
revisit those earlier results and end Subsection \ref{newencodings2} by
proposing a notion of \emph{dual coloring} derived from our algebraic
set up.  In Section \ref{preliminariesNULL} we discuss how the growth
of degree in the Nullstellensatz occurs under the assumption
$\text{P} \not= \text{NP}$~. We also sketch a linear algebra procedure we used
to compute minimum-degree Nullstellensatz certificates for particular
graphs. In Subsection \ref{HilbertNstable} we demonstrate the degree
growth of Nullstellensatz certificates for the stable set problem. In
contrast, in Subsection \ref{HilbertNcolor}, we exhibit many
non-3-colorable graphs where there is no growth of degree.

\section{Encodings}

In this section, we focus on how to find new polynomial encodings
of some combinatorial optimization problems. We begin by recalling two nice results in the polynomial
method that will be used later on. D. Bayer  established a characterization
of 3-colorability via a system of polynomial equations \cite{bayer}.
We generalize Bayer's result as follows:
\begin{lemma}\label{bayerlem}
The graph $G$ is $k$-colorable if and only if the following zero-dimensional system of equations \[
\begin{array}{rl} x_i^k-1=0, & \text{for every node $i \in V(G)$},\\
x_i^{k-1}+x_i^{k-2}x_j+\dots+x_j^{k-1}=0, & \text{for every edge $\{i,j\} \in E(G)$~,}
\end{array}
\]
has a solution. Moreover, the number of solutions equals the number of distinct $k$-colorings
multiplied by $k!$~.\label{graph_coloring_encoding}
\end{lemma}

 Recall that a {\em stable set} or {\em independent set} in a graph
 $G$ is a subset of vertices such that no two vertices in the subset
 are adjacent.  The maximum size $\alpha (G)$ of a stable set is
 called the {\em stability number} of $G$~. We view the stable sets in
 terms of their {\em incidence vectors}. These are 0/1 vectors of length
 $|V|$~, one for every stable set, where a
 one in the $i$-th entry indicates that the $i$-th vertex is a member of
 the associated stable set. These $0/1$ vectors can be
 fully described by a small system of {\em quadratic} equations:

 \begin{lemma}[L. Lov\'asz \cite{lovasz1}]
 The graph $G$ has stability number at least $k$ if and only if
 the following zero-dimensional system of equations
\begin{align}
x_i^2-x_i=0, & \quad \text{for every node $i \in V(G)$},\label{one}\\
x_ix_j=0, & \quad \text{for every edge $\{i,j\} \in E(G)$},\label{two} \\
\sum_{i=1}^n x_i=k,\label{three}
\end{align}
 has a solution.
\label{stable_set_encoding}
 \end{lemma}

\begin{exm}
\emph{
 Consider the Petersen graph labeled as in
 Figure \ref{npetersen}.
 \begin{figure}[h]
 \begin{center}
 \includegraphics[width=3.5cm, trim=0 50 0 0]{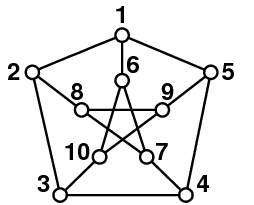}
 \end{center}
 \caption{Petersen graph}\label{npetersen}
 \end{figure}
 If we wish to check whether there are stable
 sets of size four, we take the ideal $I$ generated by the polynomials
in Eq. \ref{one}, \ref{two} and \ref{three}:
 \begin{eqnarray*}
 I &=& \big\langle x_1^2 - x_1, x_2^2 - x_2, x_3^2 - x_3, x_4^2 - x_4, x_5^2 - x_5, x_6^2 - x_6, x_7^2 - x_7, x_8^2 - x_8,
x_9^2 - x_9, x_{10}^2 - x_{10},\\
 && \hspace{5pt} x_1x_6, x_2x_8, x_3x_{10}, x_4x_7,x_5x_9,
 x_1x_2, x_2x_3, x_3x_4, x_4x_5, x_1x_5, x_6x_7,
 x_7x_8, x_8x_9, x_9x_{10}, x_6x_{10},\\
 && \hspace{5pt} x_1 + x_2 + x_3 + x_4 + x_5 + x_6 + x_7 + x_8 + x_9 + x_{10} - 4 \big\rangle.
 \end{eqnarray*}
By construction, we know that the quotient ring
 $R:=\C[x_1,\dots,x_{10}]/I$ is a finite-dimensional
 $\C$-vector space. Because the ideal $I$ is radical, its dimension
 equals the number of stable sets of cardinality four in the Petersen
 graph (not taking symmetries into account). Using Gr\"obner bases, we
 find that the monomials $1, x_{10}, x_9, x_8, x_7$ form a
 vector-space basis of $R$ and that there are no solutions with
 cardinality five; thus $\alpha(\text{Petersen})=4$~.  It is important to
 stress that we can recover the five different maximum-cardinality stable sets
 from the knowledge of the complex finite-dimensional vector space
 basis of $R$ (see \cite{coxetalII}).
 }
 \end{exm}

Next we establish similar encodings for the combinatorial problems
stated in Theorem \ref{encodingthm}.

\subsection{Proof of Theorem 1.1} \label{newencodings1}

\begin{proof}[Proof (Theorem \ref{encodingthm}, Part 1).]
Suppose that a cycle $C$ of length $L$ exists in the graph $G$~.
We set $y_i=1$ or $0$ depending on whether node $i$ is on $C$ or
not. Next, starting the numbering at any node of $C$~, we set
$x_i=j$ if node $i$ is the $j$-th node of $C$~. It is easy to check
that Eqs. \ref{sizecycle} and \ref{choosevertex} are satisfied.

To verify Eq. \ref{cyclicity}, note that since $C$ has length $L$~,
if vertex $i$ is the $j$-th node of the cycle, then one of its
neighbors, say $k$~, must be the ``follower'', namely the $(j+1)$-th
element of the cycle. If $j < L$~, then the factor $(x_i - x_k - 1) =
0$ appears in the product equation associated with the $i$-th vertex,
and the product is zero. If $j = L$~, then the factor $(x_i - x_k -
(L-1)) = 0$ appears, and the product is again 0. Since this is true
for all vertices that are turned ``on", and for all vertices that are ``off", we have Eq. \ref{cyclicity}
automatically equal to zero, all of the equations of the polynomials vanish.

Conversely, from a solution of the system above, we see that $L$
variables $y_i$ are not zero; call this set $C$~. We claim that the
nodes $i\in C$ must form a cycle. Since $y_i \not= 0$~, the polynomial
of Eq.
\ref{cyclicity} must vanish; thus for some $j \in C$~,
\begin{align*}
(x_i - x_j + 1) &= 0~, \quad \text{or} \quad (x_i - x_j - (L-1)) = 0~.
\end{align*}
Note that Eq. \ref{cyclicity} reduces to this form when $y_i = 1$~. Therefore, either vertex
$i$ is adjacent to a vertex $j$ (with $y_j = 1$) such
that $x_j$ equals the \textit{next integer value} ($x_i + 1 =
x_j$), or $x_i - L = x_j - 1$ (again, with $y_j = 1$). In the second case, since
$x_i$ and $x_j$ are integers between 1 and $L$~, this forces $x_i = L$ and
$x_j = 1$~. By the pigeonhole principle, this implies that all integer
values from 1 to $L$ must be assigned to some node in $C$ starting at
vertex 1 and ending at $L$ (which is adjacent to the node receiving
$1$).
\end{proof}

We have the following corollary.

\begin{corollary} A graph $G$ has a hamiltonian cycle if
and only if the following zero-dimensional system of $2n$ equations
has a solution. For every node $i \in V(G)$~, we have two equations:
\[
\prod_{s = 1}^n(x_i - s) = 0~, \quad \text{and} \quad \prod_{j \in Adj(i)}(x_i - x_j + 1)(x_i - x_j - (n-1)) = 0~.
\]
The number of hamiltonian cycles in the graph $G$ equals the number of solutions of the system divided by $2n$~.
\end{corollary}

\begin{proof}
Clearly when $L=n$ we can just fix all $y_i$ to $1$~, thus many
of the equations simplify or become obsolete. We only have to check
the last statement on the number of hamiltonian cycles. For that, we
remark that no solution appears with multiplicity because the ideal is
radical. That the ideal is radical is implied by the fact that every
variable appears as the only variable in a unique square-free
polynomial (see page 246 of \cite{robbiano}). Finally, note for every
cycle there
are $n$ ways to choose the initial node to be labeled as $1$~, and
then two possible directions to continue the labeling. \end{proof}

Note that similar results can be established for the directed graph
version, thus one can consider paths or cycles with
orientation. Also note that, we can use the polynomials systems above to
investigate the distribution of cycle lengths in a graph (similarly
for path lengths and cut sizes). This topic has several outstanding
questions. For example, a still unresolved question of Erd\"os and
Gy\'arf\'as  \cite{erdos} asks: If $G$ is a graph with minimum-degree
three, is it true that $G$ always has a cycle having length that is a power of
two?  Define the \emph{cycle-length polynomial} as the square-free
univariate polynomial whose roots are the possible cycle lengths of a
graph (same can be done for cuts).  Considering $L$ as a variable, the
reduced lexicographic Gr\"obner basis (with $L$ the last variable)
computation provides us with a unique univariate polynomial on $L$ that is divisible
by the cycle-length polynomial of $G$~.

Now we proceed to the proof of part 2 of Theorem \ref{encodingthm}. For
this we recall Schnyder's characterization of planarity in
terms of the dimension of a poset \cite{schnyder}: For an $n$-element
poset $P$~, a \emph{linear extension} is an order preserving bijection $\sigma:
P \rightarrow \{1,2,\dots, n\}$~. The \emph{poset dimension} of P is
the smallest integer $t$ for which there exists a family of $t$ linear
extensions $\sigma_1,\dots,\sigma_t$ of $P$ such that $x <y$ in P if
and only if $\sigma_i(x) <\sigma_i(y)$ for all $\sigma_i$~.  The
\emph{incidence poset} $P(G)$ of a graph $G$ with node set V and edge set E is
the partially ordered set of height two on the union of nodes and edges,
where we say $x < y$ if $x$ is a node and $y$ is an edge, and $y$ is
incident to $x$~.

\begin{lemma}[Schnyder's theorem \cite{schnyder}] A graph $G$ is planar if
and only if the poset dimension of $P(G)$ is no more than three.
\end{lemma}

Thus our first step is to encode the linear extensions and the poset
dimension of a poset $P$ in terms of polynomial equations. The
idea is similar to our
characterization of cycles via permutations.

\begin{lemma} The poset $P=(E,>)$ has poset dimension at most $p$ if and only if the following system of equations has a solution:

For $k=1,\dots,p:$
\begin{align} \label{llinearext1}
 \prod_{s = 1}^{|E|}(x_i(k) - s)=0, \quad \text{for every $i \in \{1,\ldots,|E|\}$}, \quad \quad \text{and} \quad
s_k\bigg(\prod_{\stackrel{\{i,j\} \in \{1,\ldots,|E|\},}{i < j}} x_i(k)-x_j(k)\bigg)=1~.
 \end{align}

For $k=1,\dots,p$~, and every ordered pair of comparable elements $e_i > e_j$ in $P$~:
\begin{align}\label{llinearext2}
 x_i(k)-x_j(k)-\Delta_{ij}(k) &=0.
 \end{align}

For every ordered pair of incomparable elements of $P$ (i.e., $e_i \not> e_j$ and $e_j \not> e_i$)~:
\begin{align} \label{nomore}
\prod_{k=1}^p \big(x_i(k)-x_j(k)- \Delta_{ij}(k) \big) &= 0~, \quad \quad
\prod_{k=1}^p \big(x_j(k)-x_i(k)- \Delta_{ji}(k)\big) = 0~,
\end{align}

For $k=1,\dots,p$~, and for every pair $\{i,j\} \in \{1,\ldots,|E|\}$:
\begin{align} \label{llinearext_diff}
\prod_{d = 1}^{|E| - 1}(\Delta_{ij}(k) -d)=0, \quad \quad \prod_{d = 1}^{|E| - 1}(\Delta_{ji}(k) -d)=0~.
\end{align}
\end{lemma}

\begin{proof}
With Eqs. \ref{llinearext1} and \ref{llinearext2}, we assign distinct numbers 1 through $|E|$ to the poset elements,
such that the properties of a linear extension are satisfied. Eqs. \ref{llinearext1} and \ref{llinearext2} are repeated
$p$ times, so $p$ linear
extensions are created. If the intersection of these extensions is indeed equal to the
original poset $P$~, then for every incomparable pair of elements in $P$~, at least one of the $p$ linear extensions must
detect the incomparability. But this is indeed the case for Eq. \ref{nomore}, which says that for the $l$-th linear extension,
the values assigned to the incomparable pair $e_i,e_j$ do not satisfy $x_i(l) < x_j(l)$~, but instead satisfy $x_j(l) > x_i(l)$~.
\end{proof}

\begin{proof}[Proof (Theorem \ref{encodingthm}, Part 2).]
We simply apply the above lemma to the particular pairs of order relations
of the incidence poset of the graph. Note that in the formulation we
added variables $z_{\{ij\}}$ that have the effect of turning on or off an edge of the input graph.
\end{proof}

\begin{exm}[Posets and Planar Graphs]
 \begin{figure}[h]
 \begin{center}
 \includegraphics[width=7cm]{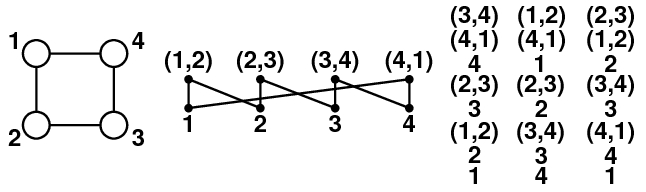}
 \end{center}
 \caption{Via Schnyder's theorem, the square is planar since $P(\text{square})$ has dimension at most three.}\label{fig_poset_graph}
\end{figure}
\emph{
}
\end{exm}

\begin{proof}[Proof (Theorem \ref{encodingthm}, Part 3).]
Using Lemma \ref{bayerlem}, we can finish the proof of Part 3. For a
$k$-colorable subgraph $H$ of size $R$~, we set $y_{ij}=1$ if edge $\{i,j\} \in E(H)$ or $y_{ij}=0$ otherwise.
By Lemma \ref{bayerlem}, the resulting
subsystem of equations has a solution. Conversely from a solution, the
subgraph $H$ in question is read off from those $y_{ij} \not= 0$~.
Solvability implies that $H$ is $k$-colorable.
\end{proof}

Before we prove Theorem \ref{encodingthm}, Part 4, we recall that the \textit{edge-chromatic number} of a
graph is the minimum number of colors necessary to color every edge of a graph such that no two edges of the
same color are incident on the same vertex.

\begin{proof}[Proof (Theorem \ref{encodingthm}, Part 4).]
If the system of equations has a solution,
then Eq. \ref{1_1_4_colors} insures that all variables $x_{ij}$ are assigned $\Delta$ roots of unity.
Eq. \ref{1_1_4_cyccity} insures that no node is incident on two edges of the same color. Since the
graph contains a vertex of degree $\Delta$~, the graph cannot have an edge-chromatic number less than
$\Delta$~, and since the graph is edge-$\Delta$-colorable, this implies that the graph has edge-chromatic
number exactly $\Delta$~. Conversely, if the graph has an edge-$\Delta$-coloring, simply map the coloring
to the $\Delta$ roots of unity and all equations are satisfied. Since Vizing's classic result shows that
any graph with maximum vertex degree $\Delta$ can be edge-colored with at most $\Delta + 1$ colors, if
there is no solution, then the graph must have an edge-chromatic number of $\Delta + 1$~.
\end{proof}

\subsection{Normal forms and Dual colorings} \label{newencodings2}

In \cite{alontarsi} Alon and Tarsi show another polynomial encoding of $k$-colorability. Here we consider one curious consequence of the polynomial method for graph
colorings when we use an algebraic encoding similar to that of
\cite{alontarsi}. By taking a closer look at the {\em normal form} of the
polynomials involved, we can derive a  notion of \emph{dual coloring}, which
has the nice property that a graph is dually $d$-colorable if and only
if it is $d$-colorable. This gives rise to an appealing new graph
invariant: the {\em simultaneous chromatic number} $\sigma(G)$~,
defined to be the infimal $d$ such that $G$ has a $d$-labeling that
is simultaneously a coloring and a dual coloring.

Fix a graph $G=(V,E)$ with $V:=\{1,\dots,n\}$ and $E\subseteq{V\choose2}$~,
fix a positive integer $d$~, and let $D:=\{0,1\dots,d-1\}$~.
Let $\alpha:=\exp({{2\pi i}\over d})\in \C$ be the primitive complex $d$-th
root of unity, so that $\alpha^0,\dots,\alpha^{d-1}$ are distinct and
$\alpha^d=1$~. For a $d$-labeling $c:V\longrightarrow D$ of the vertices of $G$~, let
$$
\e(c):=\prod \left\{(\alpha^{c(i)}-\alpha^{c(j)})\,:\, i<j,\ \{i,j\}\in E\right\}~.
$$
Clearly, $c$ is a proper $d$-coloring of $G$ if and only if $\e(c)\neq 0$~.

With every orientation $O=(V,A)$ of $G$ (where $A$ denotes the set of ``arrows" or directed edges) associate a sign $\sign^O=\pm 1$
defined by the parity of the number $|\{(i,j)\in A\,:\, i>j\}|$ of flips of
$O$ from the standard orientation (where every directed edge $(i,j)$ has $i < j$), and an out-degree vector
$\l^O:=(\l^O_1,\dots,\l^O_n)$ with $\l^O_i$ the out-degree of
vertex $i$ in $O$~. For a non-negative integer $k$ let $[k]\in D$ be the
representative of $k$ modulo $d$~, and for a vector
$\l=(\l_1,\dots,\l_n)\in V^n$ let $[\l]=([\l_1],\dots,[\l_n])\in D^V$~.
For a labeling $c^*:V\longrightarrow D$ of the vertices of $G$ let
$$
\e^*(c^*):= \sum\,\left\{\,\sign^O\ :\
\mbox{$O$ orientation of $G$ with $[\l^O]=c^*$}\, \right\}~.
$$
Call $c^*$ a {\em dual $d$-coloring} of $G$ if $\e^*(c^*)\neq 0$~.

\begin{theorem}
A graph has a $d$-coloring, namely $c\in D^V$ with $\e(c)\neq 0$
(so is $d$-colorable) if and only if it has a dual $d$-coloring,
namely $c^*\in D^V$ with $\e^*(c^*)\neq 0$ (so is dually $d$-colorable).
\end{theorem}

\begin{proof}
Let $G$ be a graph on $n$ vertices. Consider the following radical
zero-dimensional ideal $I$ in $\C[x_1,\dots,x_n]$ and its variety
$\text{variety}(I)$ in $\C^n$:
$$I \ :=\ \langle x_1^d-1,\dots,x_n^d-1\rangle \ ,\quad
\text{variety}(I) \ :=\ \{\,\alpha^c:=(\alpha^{c(1)},\dots,\alpha^{c(n)})\in\C^n
\ : \ c\in D^V \, \}\ .$$
It is easy to see that the set $\{x_1^d-1,\dots,x_n^d-1\}$
is a universal Gr\"obner basis (see \cite{BOT} and references therein). Thus, the (congruence classes of) monomials
$x^{c^*},\ c^*\in D^V$ (where $x^{c^*}:=\prod_{i=1}^n x_i^{c^*(i)}$),
which are those monomials not divisible by any $x_i^d$~, form a vector space basis
for the quotient $\C[x_1,\dots,x_n]/I$~.
Therefore, every polynomial $f=\sum a_{\l}\cdot x^{\l}$ has a unique
{\em normal form} $[f]$ with respect to this basis, namely the polynomial that
lies in the vector space spanned by the monomials $x^{c^*},\ c^*\in D^V$~,
and satisfies $f-[f]\in I$~. It is not very hard to show that this
normal form is given by $[f]=\sum a_{\l}\cdot x^{[\l]}$~.

Now consider the {\em graph polynomial} of $G$~,
$$f_G:=\prod \left\{\,(x_i-x_j)\,:\, i<j,\ \{i,j\}\in E\,\right\}~.$$
The labeling $c\in D^V$ is a $d$-coloring of $G$ if and only if
$\e(c)=f_G(\alpha^c)\neq 0$~.
Thus, $G$ is not $d$-colorable if and only if $f_G$ vanishes on
every $\alpha^c\in \text{variety}(I),$ which holds if and only if $f\in I,$ since
$I$ is radical. It follows that $G$ is $d$-colorable if and only if the
representative of $f_G$ is not zero.
Since $f_G=\sum \sign^O \cdot x^{\l^O}$~, with the sum extending
over the $2^{|E|}$ orientations $O$ of $G$~, we obtain
$$[f_G]=\sum \sign^O\cdot x^{[\l^O]}
= \sum_{c^*\in D^V}\e^*(c^*)\cdot x^{c^*}~.$$
Therefore $[f_G]\neq 0$ and $G$ is $d$-colorable if and only if
there is a $c^*\in D^V$ with $\e^*(c^*)\neq 0$~.
\end{proof}

\begin{exm}
\emph{
Consider the graph $G=(V,E)$ having $V=\{1,2,3,4\}$ and $E=\{12,13,23,24,34\}$~,
and let $d=3$~. The normal form of the graph polynomial can be shown to be
{\small
\begin{eqnarray*}
[f_G]&=& x_1^2x_2^2x_3-x_1^2x_2^2x_4+x_1^2x_2x_4^2-x_1^2x_2x_3^2+x_1^2x_3^2x_4
-x_1^2x_3x_4^2+x_1x_2-x_1x_2x_3^2x_4+x_1x_3^2x_4^2\\
&\phantom{=}&- x_1x_3+x_1x_2^2x_3x_4-x_1x_2^2x_4^2+x_3^2
-x_3x_4+x_2^2x_3x_4^2-x_2^2+x_2x_4-x_2x_3^2x_4^2\ .
\end{eqnarray*}}
\noindent
Note that in general, the number of monomials appearing in the
expansion of $f_G$ can be as much as the number of orientations $2^{|E|}$;
but usually it will be smaller due to cancellations that occur. Moreover,
there will usually be further cancellations when moving to the normal form,
so typically $[f_G]$ will have fewer monomials. In our example,
out of the $2^{|E|}=2^5=32$ monomials corresponding to the orientations,
in the expansion of $f_G$ only $20$ appear, and in the normal form
$[f_G]$ only $18$ appear due to the additional cancellation:
\begin{align*}
-[x_1x_3^3x_4]+[x_1x_2^3x_4]=-x_1x_4+x_1x_4=0~.
\end{align*}
Note that the graph $G$ in this example has only six $3$-colorings
(which are in fact the same up to relabeling of the colors), but as many as
$18$ dual $3$-colorings $c^*$ corresponding to monomials $x^{c^*}$ appearing
in $[f_G]$~. For instance, consider the labeling
$c^*(1)=c^*(2)=c^*(4)=0,\ c^*(3)=2$: the only orientation $O$ that
satisfies $[\delta^O_j]=c^*(j)$ for all $j$ is one with
edges oriented as $21,23,24,31,34$~, having $\sign^O=1$ and out-degrees
$\delta^O_1=\delta^O_4=0$~, $\delta^O_3=2$ and $\delta^O_2=3$~,
contributing to $[f_G]$ the non-zero term
$\e^*(c^*)\cdot\prod_{j=1}^4 x_j^{c^*(j)}=1\cdot x_1^0x_2^0x_3^2x_4^0=x_3^2$~.
Thus, $c^*$ is a dual $3$-coloring (but, since $c^*(1)=c^*(2)$~, it is neither
a usual $3$-coloring nor a {\em simultaneous $3$-coloring} --- see below).
}
\end{exm}

Note that in this example, and seemingly often, there are many more
dual colorings than colorings; this suggests a randomized heuristic
to find a dual $d$-coloring for verifying $d$-colorability.

A particularly appealing notion that arises is the following: call a vertex
labeling $s:V\longrightarrow D$ a {\em simultaneous $d$-coloring}
of a graph $G$ if it is simultaneously a $d$-coloring and
a dual $d$-coloring of $G$~. The {\em simultaneous chromatic number}
$\sigma(G)$ is then the minimum $d$ such that $G$ has a
simultaneous $d$-coloring. This is a strong notion that may prove useful
for inductive arguments, perhaps in the study of the 4-color problem of
planar graphs, and which provides an upper bound on the usual
chromatic number $\chi(G)$~. First note that, like the usual chromatic number,
it can be bounded in terms of the maximum degree $\Delta(G)$ as follows.

\begin{theorem}
The simultaneous chromatic number of any graph $G$ satisfies~~$\sigma(G)\leq \Delta(G)+1$~. Moreover,
for any $G$ and $d\geq\Delta(G)+1$~, there is an acyclic orientation $O$ whose out-degree vector
$\delta^O=(\delta^O_1,\dots,\delta^O_n)$ provides a simultaneous $d$-coloring $s$ defined by
$s(i):=\delta^O_i$ for every vertex $i$~.
\end{theorem}

\begin{proof} We prove the second (stronger) claim, by induction on the number $n$ of vertices.
For $n=1$, this is trivially true. Suppose $n>1$~, and let $d:=\Delta(G)+1$~. Pick any vertex $i$ of maximum
degree $\Delta(G)$~, and let $G'$ be the graph obtained from $G$ by removing vertex $i$ and all edges incident
on $i$~. Let $O'$ be an acyclic orientation of $G'$ and $s'$ the corresponding simultaneous $d$-coloring of
$G'$ guaranteed to exist by induction. Extend $O'$ to an orientation $O$ of $G$ by orienting all edges
incident on $i$ away from $i$~, and extend $s$ to the corresponding vertex labeling of $G$ by setting
$s(i):=\delta^O_i=d-1$~. Then $O$ is acyclic, and therefore $O$ is the unique orientation of $G$ with
out-degree vector $\delta^O$~. Thus,
$$\epsilon^*(s) \ = \ \sum\,\{\,\sign^{\theta}\,:\,\mbox{$\theta$ orientation of $G$ with
$[\delta^{\theta}]=s=\delta^O$}\,\}\ =\ \pm1\ \neq\ 0\ , $$and therefore $s$ is a dual
$d$-coloring of $G$~. Moreover, if $j$ is any neighbor of $i$ in $G$~, then the degree of
$j$ in $G'$ is at most $d-2$~, and therefore its label $s'(j)=\delta^{O'}(j)\leq d-2$~, and
hence $s(j)=s'(j)\neq d-1= s(i)$~. Therefore, $s$ is also a $d$-coloring of $G$~, completing the induction.
\end{proof}

\begin{exm}[simultaneous 4-coloring of the Petersen graph]\\
\emph{
According to Figure \ref{prop_sim_chrom_ex}, $\delta^O=(2,1,0,2,0,3,1,2,3,1)$~.
By inspection of Figure \ref{prop_sim_chrom_ex}, $s(i) := \delta_i^O$ does indeed describe a valid 4-coloring of the Petersen graph.
}
 \begin{figure}[h]
 \begin{center}
 \includegraphics[width=7cm]{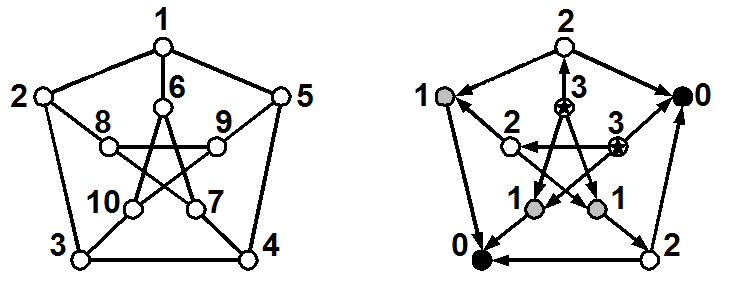}
 \end{center}
 \caption{Left: A vertex labeling. Right: An acyclic orientation labeled with out-degrees.}\label{prop_sim_chrom_ex}
 \end{figure}
\end{exm}

There are many fascinating new combinatorial and computational problems
related to this new graph invariant, the behavior of which is quite different
from that of the usual chromatic number. For instance, the direct analog
of Brooks' theorem, which states that every connected graph with
maximum degree $\Delta$ that is neither complete nor an odd cycle
is $\Delta$-colorable, fails: It is not hard to verify that the
simultaneous chromatic number of the cycle $C_n$ is $2$ if and only
if $n$ is a multiple of $4$; thus, the hexagon satisfies
$\sigma(C_6)=3>\Delta(C_6)$~.
Which are the simultaneous chromatic Brooks graphs, i.e. those with
$\sigma(G)=\Delta(G)$ ? What is the complexity of deciding if a graph is
simultaneously $d$-colorable? Which graphs are simultaneously $d$-colorable
for small $d$~? For $d=2$~, the complete answer was given by L. Lov\'asz \cite{LovOb}
during a discussion at the Oberwolfach Mathematical Institute:

\begin{theorem}\label{Bipartite}{\bf (Lov\'asz)}
A connected bipartite graph $G=(A,B,E)$
has simultaneous chromatic number $\sigma(G)=2$ if and only if
at least one of $|A|$ and $|B|$ has the same parity as $|E|$~.
\end{theorem}

\section{Nullstellensatz Degree Growth in Combinatorics} \label{preliminariesNULL}

The Hilbert Nullstellensatz states that a system of polynomial
equations $\left\{f_1(x)=0,f_2(x)=0,\dots,f_r(x)=0\right\} \subseteq \mathbb{C}[x_1,\ldots,x_n]$ has no solution in ${\mathbb C}^n$ if and only if
there exist polynomials $\alpha_1,\dots,\alpha_r \in {\mathbb
  C}[x_1,\dots,x_n]$ such that $1=\sum \alpha_if_i$ (see
\cite{coxetal}).  The purpose of this section is to investigate the
degree growth of the coefficients $\alpha_i$~. In particular,
systems of polynomials coming from combinatorial optimization.

In our investigations, we will often need to find explicit
Nullstellensatz certificates for specific graphs. This can be done via
linear algebra. First, given a system of polynomial equations, fix a
tentative degree for the coefficient polynomials $\alpha_i$ in the
Nullstellensatz certificates. This yields a \emph{linear} system of
equations whose variables are the coefficients of the monomials of the
polynomials $\alpha_1,\dots,\alpha_r$~. Then, solve this linear
system. If the system has a solution, we have found a
Nullstellensatz certificate. Otherwise, try a higher degree for the polynomials $\alpha_i$~.  For the
Nullstellensatz certificates, the
degrees of the polynomials $\alpha_i$ cannot be more than known bounds (see e.g.,
\cite{kollar} and references therein), thus this is a finite (but
potentially long) procedure to decide whether a system of polynomials is feasible
or not. In practice, sometimes low degrees suffice to find a certificate.

\begin{exm}
\emph{
Suppose we wish
 to test $K_4$ for 3-colorability, and we assume that the $\alpha_i$ in the Nullstellensatz
 certificate have degree 1. After encoding $K_4$ with the system of polynomial equations, we
``conjecture" that there exists a Nullstellensatz certificate of the following form
 \begin{align*}
 1 &= (c_{1}x_1 + c_{2}x_2 + c_{3}x_3 + c_{4}x_4 + c_{5})(x_1^3 - 1) + (c_{6}x_1 + c_{7}x_2 + c_{8}x_3 + c_{9}x_4 + c_{10})(x_2^3 - 1)\\
 &\quad + (c_{11}x_1 +\cdots + c_{15})(x_3^3 - 1)  + (c_{16}x_1 +\cdots +c_{20})(x_4^3 - 1)\\
 &\quad + (c_{21}x_1 +\cdots +c_{25})(x_1^2 + x_1x_2 + x_2^2) + (c_{26}x_1 +\cdots+ c_{30})(x_1^2 + x_1x_3 + x_3^2) \\
 &\quad + (c_{31}x_1 +\cdots +c_{35})(x_1^2 + x_1x_4 + x_4^2) + (c_{36}x_4 +\cdots+ c_{40})(x_2^2 + x_2x_3 + x_3^2) \\
 &\quad + (c_{41}x_1 +\cdots + c_{45})(x_2^2 + x_2x_4 + x_4^2) + (c_{46}x_1 +\cdots + c_{50})(x_3^2 + x_3x_4 + x_4^2).
 \end{align*}
 When we multiply out this certificate, we group
 together like powers of $x_1,x_2,x_3,x_4$ as follows:
 \begin{align*}
 1 &= c_{1}x_1^4 + \cdots + c_{13}x_3^4 + \cdots + c_{8}x_2^3x_3 + \cdots  + (c_{22} + c_{21} + c_{27} + c_{32})x_1^2x_2 + \cdots \\
& \quad + (c_{35} + c_{45} + c_{50})x_4^2 + \cdots + (-c_{15} - c_{20} - c_{5} - c_{10}).
 \end{align*}
 Because the Nullstellensatz certificate is identically 1, this
 identity gives rise to the following system of linear equations: $0 = c_{1},
 0 = c_{13} , 0 = c_8, 0 = c_{22} + c_{21} + c_{27} + c_{32}, 0 =
 c_{35} + c_{45} + c_{50}, \dots, 1 = -c_{15} - c_{20} - c_{5} -
 c_{10}$~. In other words, we have a large-scale sparse 
 system of linear equations that consists only of 1s and $-1$s. We
 implemented an exact-arithmetic linear system solver.
In this example, it turns out that degree 1 is not sufficient for
generating a Nullstellensatz certificate --- that is, this linear
system has no solution. Ultimately, we discovered that degree four is
required, and we were able to produce the following certificate:
\small
\begin{align*}
1 &= (-x_1^3-1)(x_1^3-1)+\bigg(\frac{4}{9}x_4^4-\frac{5}{9}x_4^3x_2-\frac{2}{9}x_4^3x_3-\frac{4}{9}x_4^3x_1+\frac{2}{9}x_4^2x_2x_1+\frac{2}{9}x_4^2x_3x_1\bigg)(x_4^2+x_2x_4+x_2^2)\\
&\phantom{=} + \bigg(\frac{1}{9}x_4^4+\frac{2}{9}x_4^3x_2-\frac{1}{9}x_4^3x_1-\frac{2}{9}x_4^2x_2x_1 \bigg)(x_2^2+x_3x_2+x_3^2)  + \bigg(\frac{2}{9}x_4^4+\frac{1}{9}x_4^3x_2+\frac{1}{9}x_4^3x_1+\frac{2}{9}x_4^2x_2x_1 \bigg)(x_4^2+x_3x_4+x_3^2) \\
&\phantom{=} + \bigg(-\frac{2}{3}x_4^4+x_4^3x_1-x_4x_1^3+x_1^4\bigg)(x_4^2+x_1x_4+x_1^2)+\frac{1}{3}x_4^3x_2(x_2^2+x_1x_2+x_1^2) +\bigg(-\frac{1}{3}x_4^4-\frac{1}{3}x_4^3x_2\bigg)(x_3^2+x_1x_3+x_1^2).
\end{align*}
\normalsize
}
 \end{exm}

\begin{lemma}  If $\text{P} \neq \text{NP}$~, then there must exist an
infinite family of graphs whose minimum-degree non-3-colorability
Nullstellensatz certificates have unbounded growth with respect to the
number of vertices and edges in the graph.
 \end{lemma}

\begin{proof} Our proof is by contradiction with the hypothesis $\text{P} \neq \text{NP}$~. Consider a
 non-3-colorable graph that has been encoded as the system of
 polynomial equations $(x_i^3 - 1) = 0$ for $i \in V(G)$~, and $
 (x_i^2 + x_ix_j +x_j^2) = 0$ for $\{i,j\} \in E(G)$~. Assume that every minimum-degree non-3-colorability
 Nullstellensatz certificate
has $\deg(\alpha_i) < d$ for some constant $d$~. We will show that $\text{P} = \text{NP}$ by providing a
polynomial-time algorithm for solving Graph-3-Coloring: (1) Given a graph $G$~, encode it as the above
system of polynomial equations, (2) Construct and solve the associated linear system for monomials of
degree $< d$~, (3) If the system has a solution, a Nullstellensatz certificate exists, and the graph is
non-3-colorable: Return \textbf{no}, (4) If the system does \textit{not} have a solution, there does
\textit{not} exist a Nullstellensatz certificate, and the graph is 3-colorable: Return \textbf{yes}.

Now we analyze the running time of this algorithm. In Step 1, our encoding has one polynomial equation
per vertex and one polynomial equation per edge. Since there are $O(n^2)$ edges in a graph, our polynomial
system has $n + n^2 = O(n^2)$ equations. Since every equation only contains coefficients $\pm 1$ and is of
degree three or less, encoding the graph as the above system of polynomial equations clearly runs in polynomial-time.

For Step 2, we note that by Corollary 3.2b of \cite{schrijver}, if a system of linear equations $Ax=b$ has a solution,
then it has a solution polynomially-bounded by the bit-sizes of the matrix $A$ and the vector $b$
(see \cite{schrijver} for a definition of bit-size). In this case, the vector $b$ contains only
zeros and ones. To calculate the bit-size of $A$~, we recall our assumption that, for every $\alpha_i$~,
$\deg(\alpha_i) < d$ for some constant $d$~. Therefore, an upper bound on the number of terms in each
$\alpha_i$ is the total number of monomials in $n$ variables of degree less than or equal to $d$~.
Therefore, the number of terms in each $\alpha_i$ is
\begin{align}
\binom{n + d -1}{n-1} + \binom{n + d -2}{n-1} + \cdots + \binom{n  -1}{n-1}
&= O(n^d) + O(n^{d-1}) + \cdots +O(1)
= O(n^d). \nonumber
\end{align}
Since there are $O(n^2)$ equations, there are at most $O(n^{d+2})$ unknowns in the linear system, and thus, $O(n^{d+2})$ columns in $A$~.
Since the vertex equations $(x_i^3 - 1 )= 0$ have two terms, and the edge
 equations $(x_i^2 + x_ix_j +x_j^2) = 0$ have 3 terms, there are $O(n^{d + 2})$ terms in the
 \textit{expanded} Nullstellensatz
certificate, and $O(n^{d + 2})$ rows in $A$~. Since entries in $A$ are $0,\pm 1$~, the matrix $A$
contains only entries of bit-size at most 2. Therefore, the bit-sizes of both $A$ and $b$ are
polynomially-bounded in $n$~, and by Theorem 3.3 of \cite{schrijver}, the linear system can be solved in polynomial-time.

Therefore, we have demonstrated a polynomial-time algorithm for solving Graph-3-Coloring, and
since Graph-3-Coloring is NP-Complete (\cite{garey_and_johnson}), this implies $\text{P} = \text{NP}$~,
which contradicts our hypothesis. Therefore, $\deg(\alpha_i) \nleq d$ for any constant $d$~. \end{proof}

Thus, in the linear algebra approach to finding a minimum-degree
 Nullstellensatz certificate, the existence of a
 universal constant bounding the degree is
 impossible under a well-known hypothesis of complexity theory.  Clearly, a similar result can be
 obtained for other encodings (for a generalized statement see \cite{susanthesis}). Note that the linear algebra method does not rely on any property that is
 \textit{unique} to a particular combinatorial or NP-Complete problem;
 the only assumption is that the problem can be \textit{represented}
 as a system of polynomial equations. We will use it to find
 Nullstellensatz certificates of non-3-colorability and sizes of
 stable sets of graphs.

We remark that this linear algebra method finds not only a Nullstellensatz certificate (if it
 exists), but it finds one of {\em minimum possible degree}.  With our
 implementation, we ran several experiments. We quickly found out that the
 systems of linear equations are numerically unstable, thus it is best
 to use exact arithmetic to solve them. The systems of linear
 equations are also quite large in practice, as the bound on the degree
 of the polynomial coefficients grows.  Thus we need ways to reduce
 the number of unknowns.

 We will not discuss here ad hoc methods that depend on the particular
 polynomial system at hand (see \cite{susanthesis} for methods
 specific to $3$-colorability), but one rather useful general trick is
 to randomly eliminate variables in the above procedure. Instead of
 allowing \emph{all} monomials of degree $\leq d$ to appear in the
 construction of the linear system of equations, we can randomly set
 unknowns in the linear system of equations to $0$ --- e.g., set each variable to 0 with probability $p$~,
 independently, to get a
 smaller system.

 \begin{figure}[h]
 \begin{center}
 \includegraphics[width=12cm, trim=0 0 0 0]{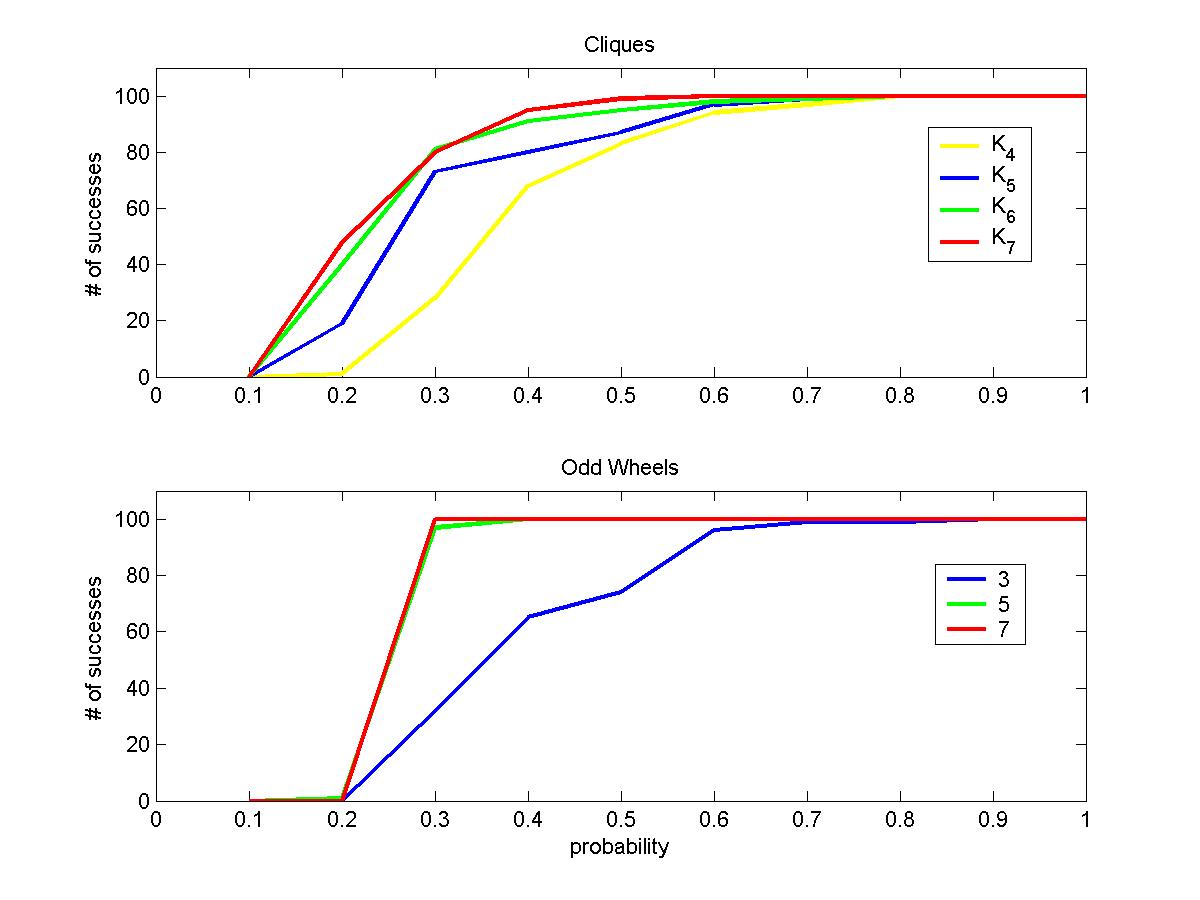}
 \caption{Probability tests on cliques and odd wheels.}
 \label{fig_prob_K_W}
 \end{center}
 \end{figure}

 In Figure \ref{fig_prob_K_W}, we see the results of the probabilistic
 search for smaller Nullstellensatz certificates. On the $x$-axis is
 the probability $p$ of keeping an unknown in the linear system. Thus,
 if $p = 0.1$~, 90\% of the time we set the unknown to 0, and only 10\%
 of the time, we keep it in the system. For the cliques and odd
 wheels, we know that there is always a certificate of degree four. For
 every probability $0.1, 0.2,\ldots,1$ we performed 100 searches for a
 degree four certificate. For both the cliques and the odd wheels, at
 $p=0.1$ and $p=0.2$~, we almost never found a certificate. But for
 $p=0.4$~, we found certificates 95\% of the time. In practice this
 idea is useful. We can reduce the number of variables in the linear
 system by 60\%, and still find a Nullstellensatz certificate 90\% of
 the time. We will report the results of specific computations in Section
 \ref{HilbertNcolor}.

 \subsection{The Nullstellensatz and Stable Sets of Graphs} \label{HilbertNstable}

Lov\'asz \cite{lovasz1} stated the challenge of finding an explicit family of graphs with growth in the degree of Nullstellensatz certificates. Here we solve his challenge. Our main result is stated in Theorem 1.2: For every graph $G$~, there
exists a Nullstellensatz certificate of degree $\alpha(G)$~, the
stability number of $G$~, certifying that $G$ has no stable set of
size $> \alpha(G)$~. Moreover, this is the minimum possible degree
for all graphs. In what follows, for any graph $G$ with stability
number $\alpha(G)$ and integer $r \geq 1$~, we associate the infeasible
system of polynomial equations $J(G,r)$~:
\begin{eqnarray}
 && -(\alpha(G)+r) + \sum_{i=1}^n x_i=0~, \label{c}\\
 &&x_i^2-x_i=0, \mbox{ for every node } i \in V(G)~,\label{b}\\
 &&x_ix_j=0, \hbox{ for every edge } \{i,j\} \in E(G)~.\label{a}
 \end{eqnarray}
Thus, the Nullstellensatz certificate will have the general form:
\begin{equation} \label{stableHN}
  1= A \bigg( -(\alpha(G)+r) + \sum_{i=1}^n x_i \bigg) + \sum_{i \in V(G)} Q_i(x_i^2-x_i)+ \sum_{\{i,j\} \in E(G)} Q_{ij} (x_ix_j)~.
\end{equation}
In the rest of this section, we will refer to the coefficient polynomials using
these particular letters (i.e. $A, Q_i$~, etc).

\begin{defi} A Nullstellensatz certificate (Eq. \ref{stableHN})
\textit{ has degree $d$} if $\max\{deg(A),deg(Q_i),deg(Q_{ij})\}=d$~.
\end{defi}

\begin{lemma} \label{reduced} For any graph $G$ and a
Nullstellensatz certificate
\begin{align}
1 &= A \underbrace{\bigg(-(\alpha(G)+r) + \sum_{i=1}^n x_i \bigg)}_{B} + \sum_{i \in V(G)}
Q_i (x_i^2-x_i)+ \sum_{\{i,j\} \in E(G)} Q_{ij} (x_ix_j) \label{reduction_lemma_orig_null}
\end{align}
certifying that $G$ has no stable set of size $(\alpha(G)+r)$ (with $r \geq 1$), we can
construct a ``reduced'' Nullstellensatz certificate
\begin{align*}
1 &= A'\bigg(-(\alpha(G)+r) + \sum_{i=1}^n x_i \bigg) + \sum_{i \in V(G)}
Q'_i (x_i^2-x_i)+ \sum_{\{i,j\} \in E(G)} Q'_{ij} (x_ix_j),
\end{align*}
such that
\begin{enumerate}
 \renewcommand{\theenumi}{\arabic{enumi}.}

\item The coefficient $A'$ multiplying
$-(\alpha(G)+r) + \sum_{i=1}^n x_i$ has only square-free monomials
supported on stable sets of $G$~, and thus $\deg(A') \leq \alpha(G)$~.

\item $\max\{deg(A),deg(Q_i),deg(Q_{ij})\}=\max\{deg(A'),deg(Q'_i),
deg(Q'_{ij})\}$~. Thus, if the original Nullstellensatz certificate
has  minimum-degree, the ``reduced" certificate also has
minimum-degree.
\end{enumerate}
\end{lemma}

\begin{proof} Let $I$ be the ideal generated by $
x_i^2-x_i$ (for every node  $i \in V(G)$), and $x_ix_j$ ( for every
edge $\{i,j\} \in E(G)$). We apply reductions modulo $I$ to Eq.
\ref{reduction_lemma_orig_null}. If a non-square-free monomial
appears in polynomial $A$~, say
$x_{i_1}^{\alpha_1}x_{i_2}^{\alpha_2}\cdots x_{i_k}^{\alpha_k}$ with
at least one $\alpha_j>1$~, then we can subtract the polynomial
$x_{i_1}^{\alpha_1}x_{i_2}^{\alpha_2}\cdots,x_{i_j}^{\alpha_j-2}x_{i_k}^{\alpha_k}
B(x_{i_j}^2-x_{i_j})$ from $AB$ and simultaneously add it to $\sum
Q_{s}(x_{s}^2-x_{s})$~. Thus, eventually we obtain a new certificate
that has only square-free monomials in $A'$~. Furthermore, if
$Q'_{s}$ has new monomials, they are of degree less than or equal to
what was originally in $A$~.

Similarly, if $x_{i_1}x_{i_2}\cdots x_{i_k}$ appears in $A$~,
but $x_{i_1}x_{i_2}\cdots x_{i_k}$ contains an edge $\{i,j\}
\in E(G)$ (if $x_ix_j$ divides $x_{i_1}x_{i_2}\cdots
x_{i_k}$), then we can again subtract
$B(x_{i_1}x_{i_2}\cdots x_{i_k}/x_ix_j)(x_ix_j)$ from $AB$~, and, at the
same time, add it to $\sum_{\{i,j\} \in E(G)}Q_{ij} x_ix_j$~. Furthermore, the
degree is maintained, and we have reached the form we claim exists
for $A'$~. \end{proof}

We now show that, for \textit{every} graph, there exists an
explicit Nullstellensatz certificate of degree $\alpha(G)$~.

\begin{theorem} Given a graph $G$~, there exists a Nullstellensatz certificate of
degree $\alpha(G)$ certifying the non-existence of stable
sets of size greater than $\alpha(G)$~. \label{thm_ss_cert_exists}
\end{theorem}

\begin{proof}
The proof is an algorithm to construct the explicit Nullstellensatz
certificate for the non-existence of a stable set of size $\alpha(G) + r$, with $r \geq 1$~.
First, let us establish some notation: Let $S(i,G)$ be the
set of all stable sets of size $i$ in $G$~. We index nodes in the
graph by integers, thus stable sets in $G$ are subsets of
integers. When we refer to a monomial $x_{d_1}x_{d_2}\cdots x_{d_i}$
as a ``stable set", we mean $\{d_1,\ldots,d_i\}
\in S(i, G)$~.  We use ``hat" notation to remove a variable
from a monomial, meaning $x_{d_1}x_{d_2}\cdots x_{d_{i-1}}\widehat{x_{d_i}}x_{d_{i+1}}\cdots x_{d_k} =
x_{d_1}x_{d_2}\cdots x_{d_{i-1}}x_{d_{i+1}}\cdots x_{d_k}$~.
Finally, let
\begin{align*}
\mathscr{P}(i, G) &:= \sum_{\{d_1, d_2,\ldots,d_i\} \in S(i,G)}x_{d_1}x_{d_2}\cdots x_{d_i}~, \quad \quad \quad \text{and} \quad \mathscr{P}(0, G) := 1~.
\end{align*}
When $\mathscr{P}(i, G)$
is ordered lexicographically, we denote $\mathscr{P}_j(i,G)$ as the
$j$-th term. We also define the constants
\begin{align*}
C_i^G &:= \frac{iC_{i-1}^G}{\alpha(G) + r - i}~, \quad \quad \text{and}  \quad C_0^G := \frac{1}{\alpha(G) + r}~.
\end{align*}

Our main claim is that we can construct explicit coefficients $A, Q_i, Q_{ij}$ of degree less than or
equal to $\alpha(G)$ such that the following identity is satisfied
\begin{align*}
1 &=A \underbrace{
\bigg(-(\alpha(G) + r) + \sum_{i=1}^n x_i \bigg)}_{B} + \underbrace{\sum_{\{i,j\} \in E(G)}Q_{ij}x_ix_j + \sum_{i=1}^nQ_i(x_i^2 - x_i)}_{C}~.
\end{align*}
We do so by the algorithm of Figure \ref{NCC}.
In Figure \ref{NCC} and in what follows, $AB$ and $C$ refer to parts of the equation as marked above.

\begin{figure}
\begin{tabbing}
***\=***\=***\=***\=***\=***\=***\=***\=***\=***\=***\=***\=********\\
$\text{ALGORITHM (Nullstellensatz certificate construction)}$\\
$\text{INPUT: A graph}\ G(V,E), \ \text{associated polynomials $\mathscr{P}(i,G)$~,}\ \alpha(G),r $\\
$\text{OUTPUT: polynomials }(A, Q_i, Q_{ij}) \ \text{such that}\ AB+C=1 \text{ is true} $\\
0 \> $A \leftarrow -\sum_{i=0}^{\alpha(G)} C_i^G \mathscr{P}(i, G)$\\
1 \> $Q_i \leftarrow 0$~, for $i=1\ldots |V|$\\
2 \> $Q_{ij} \leftarrow 0$~, for $\{i,j\} \in E(G)$\\
3\> \textbf{for} $i \leftarrow 0 \textbf{ to } \alpha(G)$\\
4\>\> \textbf{for} $j \leftarrow 1 \textbf{ to } \text{\# of monomials in $\mathscr{P}(i,G)$}$\\
5\>\>\> \textbf{for} $k \leftarrow 1 \textbf{ to } |V|$\\
6\>\>\>\> Let $\mathscr{P}_j(i, G) = x_{d_1}x_{d_2}\cdots x_{d_{i}}$\\
7\>\>\>\> \textbf{if} $\mathscr{P}_j(i, G)x_k$ is a square-free stable set\\
8\>\>\>\>\> (rule 1) $Q_{k} \leftarrow Q_{k} + C_{i+1}^Gx_{d_1}x_{d_2}\cdots x_{d_{i}}$\\
9\>\>\>\> \textbf{else if} $\mathscr{P}_j(i, G)x_k$ is a square-free non-stable set\\
10\>\>\>\>\>Choose an edge $l$, $1\le l\le i$, so that $(k,d_l)$ is an edge of $G$ \\
12\>\>\>\>\>(rule 2) $Q_{kd_l} \leftarrow Q_{kd_l} + C_i^Gx_{d_1}x_{d_2}\cdots x_{d_{l-1}}\widehat{x_{d_l}}x_{d_{l+1}}\cdots x_{d_{i}}$\\
13\>\>\>\> \textbf{end if}\\
14\>\>\> \textbf{end for}\\
15\>\> \textbf{end for}\\
16\> \textbf{end for}\\
17\> \textbf{return} $A, Q_i, Q_{ij}$\\
********************************************
\end{tabbing}
\caption{Nullstellensatz certificate construction}\label{NCC}
\end{figure}

It is clear that the algorithm terminates for any finite graph,
since the number of iterations of each  \textbf{for} loop is finite.
To prove correctness of the algorithm, we demonstrate that the following
statement is true for each iteration of the \textbf{for}
loop beginning in line 3:

\begin{align*}
&\hbox{Prior to the $m$-th iteration, $AB+C - 1$ only contains terms of degree $m+1$ or greater.}\atop
\hbox{Furthermore, all terms of degree $m+1$ are square-free.}\tag{$*$}
\end{align*}

When the \textbf{for} loop is initialized, $A$ is set to the
polynomial $-\sum_{i=0}^{\alpha(G)} C_i^G \mathscr{P}(i, G)$~, and
$Q_i, Q_{ij}$ are set to
zero. Therefore, prior to the $0$-th iteration, $C = 0$ and $AB + C -
1 = AB - 1$~. Since the constant term in $AB$ is equal to
\begin{align*}
-C_0^G\mathscr{P}(0,G)\big( -(\alpha(G) + r)\big) &= -\frac{1}{\alpha(G) + r}\bigg( -(\alpha(G) + r) \bigg) = 1~,
\end{align*}
this implies that $AB - 1$ only contains terms of degree 1 or higher.
Furthermore, linear terms are by definition square-free.
Thus, $(*)$ is true at initialization.

Now we must show that if $(*)$ is true prior to the $m$-th
iteration of the \textbf{for} loop, then it will be true prior to the
$(m + 1)$-th iteration. Assume then, that prior to the $m$-th
iteration, $AB + C - 1$ only contains terms of degree $m+1$ or
greater, and that all terms of degree $m+1$ are square-free. We must
show that prior to the $(m+1)$-th iteration, $AB+C-1$ only contains
terms of degree $m+2$ or greater, and furthermore, that all terms of
degree $m+2$ are square-free. Prior to the $m$-th iteration, there
are only two kinds of terms of degree $m+1$ in $AB + C -1$~: (1) terms corresponding to square-free, stable
sets, and (2) terms corresponding to square-free, non-stable sets. We
will show that both kinds of terms cancel during the $m$-th
iteration.
\begin{itemize}
\item Let $x_{d_1}x_{d_2}\cdots x_{d_{m+1}}$ be any square-free, stable
set monomial in $AB + C - 1$ of degree $(m+1)$~. Since
$\{d_1,d_2,\ldots,d_{m+1}\}$ is a stable set, all subsets of
size $m$ are likewise stable sets and appear as summands in
$\mathscr{P}(m,G)$~. Consider the coefficient of $x_{d_1}x_{d_2}\cdots
x_{d_{m+1}}$ in $C$~. During the $m$-th iteration, we apply rule 1 and
create this monomial $\binom{m+1}{m}$ times in $C$~. Since this
monomial is created by the multiplication of $x_{d_1}x_{d_2}\cdots
\widehat{x_{d_{k}}} \cdots x_{d_{m+1}}$ with $(x^2_{d_k} - x_{d_k})$~,
the coefficient for $x_{d_1}x_{d_2}\cdots x_{d_{m+1}}$ in $C$ is
\begin{align*}
-\binom{m+1}{m}C_{m+1}^G &= -(m+1)C_{m+1}^G
\end{align*}
Now we will calculate the coefficient of this same monomial in $AB$~.
This monomial is created in two ways, (1) multiplying
$-C_m^Gx_{d_1}x_{d_2}\cdots
x_{d_{j-1}}\widehat{x_{d_j}}x_{d_{j+1}}\cdots x_{d_{m+1}}$ by
$x_{d_j}$ (repeated $(m+1)$ times, once for each $x_{d_j}$), or (2)
multiplying $-C_{m+1}^Gx_{d_1}x_{d_2}\cdots x_{d_{m+1}}$ by
$-(\alpha(G) + r)$ (occurring exactly once). Therefore, the
coefficient of $x_{d_1}x_{d_2}\cdots x_{d_{m+1}}$ in $AB$ is
\begin{align*}
&-C_{m+1}^G \big( -(\alpha(G) + r)\big) - (m+1)C_m^G\\
&\qquad=\alpha(G)C_{m+1}^G + rC_{m+1}^G -\frac{(m+1)C_m^G}{\alpha(G) + r - (m+1)}(\alpha(G) + r - (m+1))\\
&\qquad= \alpha(G)C_{m+1}^G + rC_{m+1}^G -C_{m+1}^G\big(\alpha(G) + r - (m+1)\big)\\
&\qquad= \alpha(G)C_{m+1}^G + rC_{m+1}^G -\alpha(G)C_{m+1}^G - rC_{m+1}^G + (m+1)C_{m+1}^G\\
&\qquad= (m+1)C_{m+1}^G~.
\end{align*}
Therefore, the coefficient for any square-free stable-set monomial in $AB + C - 1$ is
\begin{align*}
\underbrace{(m+1)C_{m+1}^G}_{\text{from $AB$}} \underbrace{-(m+1)C_{m+1}^G}_{\text{from $C$}} &= 0~.
\end{align*}

\item
Now, consider any square-free non-stable-set monomial
$x_{d_1}x_{d_2}\cdots x_{d_{m+1}}$ in $AB + C - 1$ of degree
$m+1$~. Consider all $\binom{m+1}{m}$ subsets of
$\{d_1,d_2,\ldots,d_{m+1}\}$~, and let $M$ be the number of stable sets among those $\binom{m+1}{m}$ subsets.
Each of those $M$
subsets appears as a summand in $\mathscr{P}(m, G)$~. Therefore, the
monomial $x_{d_1}x_{d_2}\cdots x_{d_{m+1}}$ is created $M$ times in
$AB$~, and $M$ times in $C$~, by $M$ applications of rule 2~. Therefore,
the coefficient for any square-free non-stable set
monomial $x_{d_1}x_{d_2}\cdots x_{d_{m+1}}$ in $AB + C - 1$ is
\begin{align*}
\underbrace{-MC_{m}^G}_{\text{from $AB$}} + \underbrace{MC_{m}^G}_{\text{from $C$}} &= 0~.
\end{align*}
\end{itemize}

Finally, consider any non-square-free monomial
$x_{d_1}x_{d_2}\cdots x^2_{d_l}\cdots x_{d_{m+1}}$ in $AB + C - 1$ of
degree $m+2$~.  We note that $\{d_1, d_2, \ldots, d_{m+1}\}$ is an
stable set. To see this, note that every non-square-free
monomial in $AB$ is created by the product of an stable set with
a linear term, and every non-square-free monomial in $C$ is
created by the product of an stable set with $(x_{d_l}^2 -
x_{d_l})$ for some $d_l$~. During the $m$-th iteration, during
applications of rule 1, $x_{d_1}x_{d_2}\cdots \widehat{x_{d_l}}\cdots
x_{d_{m+1}}$ is added to $Q_{d_l}$~. Therefore, when $Q_{d_l}$ is
subsequently multiplied by $(x_{d_l}^2 - x_{d_l})$~, the monomial
$x_{d_1}x_{d_2}\cdots x^2_{d_l}\cdots x_{d_{m+1}}$ is created. To
summarize, this monomial is created in only one way in $AB$ and only
one way in $C$~. Therefore, the coefficient for $x_{d_1}x_{d_2}\cdots
x^2_{d_l}\cdots x_{d_{m+1}}$ in $AB + C - 1$ is
\begin{align*}
\underbrace{-C_{m+1}^G}_{\text{from $AB$}} + \underbrace{C_{m+1}^G}_{\text{from $C$}} &= 0~.
\end{align*}

Therefore, we have proven that $(*)$ is valid prior to every
iteration of the \textbf{for} loop. Finally, upon termination, $i = \alpha(G) +
1$~. Therefore, by $(*)$, we know that $AB + C - 1$ only
contains monomials of degree $\alpha(G) + 2$ or greater. However,
there are no terms of degree $\alpha(G) + 2$ in $AB$ since $\deg(AB)
= \alpha(G) + 1$~.  Additionally, during the $\alpha(G)$-th iteration,
$\mathscr{P}(\alpha(G),G)x_k$ is never an stable set. Therefore,
only applications of rule 2 occur during the last iteration, and the
final degrees of $Q_i,Q_{ij}$ are less than or equal to $\alpha(G) -
1$~. Therefore, the monomial in $C$ of greatest degree is of degree $\alpha(G) +
1$~. Thus, there are no monomials in $AB + C - 1$ of degree
$\alpha(G) + 2$ or greater, and upon termination, $AB + C - 1 = 0$~.

We have then shown that we can construct $A, Q_i, Q_{ij}$ such that
\begin{align*}
1 &= \underbrace{
\bigg(-\sum_{i=0}^{\alpha(G)} C(i, G) \mathscr{P}(i, G) \bigg)}_{A}
\underbrace{
\bigg(-(\alpha(G) + r) + \sum_{i=1}^n x_i \bigg)}_{B} + \underbrace{\sum_{\{i,j\} \in E(G)}Q_{ij}x_ix_j + \sum_{i=1}^nQ_i(x_i^2 - x_i)}_{C}~.
\end{align*}
Since $\deg(Q_i), \deg(Q_{ij}) \leq (\alpha(G) - 1)$~, and $\deg(A) = \alpha(G)$~, this concludes our proof. \end{proof}

\begin{exm}
\emph{
We display a certificate found by our algorithm. In Figure \ref{T53} is
the Tur\'an graph $T(5,3)$~. It is clear that $\alpha(T(5,3)) =
2$~. Therefore, we ``test'' for a stable set of size 3.
\begin{figure}[h]
\begin{center}
\includegraphics[width=4cm, trim=0 0 0 0]{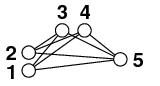}
\end{center}
\caption{Tur\'an graph $T(5,3)$}\label{T53}
\end{figure}\\
The certificate constructed by our algorithm is
\begin{align*}
1 &= \bigg(-\frac{1}{3}\big(x_1x_2 + x_3x_4\big) - \frac{1}{6}\big(x_1 + x_2 + x_3 + x_4 + x_5\big) - \frac{1}{3} \bigg)(x_1 +  x_2 + x_3 + x_4 + x_5 - 3)+\\
& \hspace{14pt} \bigg(\frac{1}{3}x_4 + \frac{1}{3}x_2 + \frac{1}{3}\bigg)x_1x_3 + \bigg(\frac{1}{3}x_2 + \frac{1}{3}\bigg)x_1x_4 + \bigg(\frac{1}{3}x_2 + \frac{1}{3}\bigg)x_1x_5 + \bigg(\frac{1}{3}x_4 + \frac{1}{3}\bigg)x_2x_3 + \\
& \hspace{14pt}\bigg(\frac{1}{3}\bigg)x_2x_4 + \bigg(\frac{1}{3}\bigg)x_2x_5 + \bigg(\frac{1}{3}x_4 + \frac{1}{3}\bigg)x_3x_5 + \bigg(\frac{1}{3}\bigg)x_4x_5 + \bigg(\frac{1}{3}x_2 + \frac{1}{6}\bigg)(x_1^2 - x_1) + \\
&\hspace{14pt} \bigg(\frac{1}{3}x_1 + \frac{1}{6}\bigg)(x_2^2 - x_2) + \bigg(\frac{1}{3}x_4 + \frac{1}{6}\bigg)(x_3^2 - x_3) + \bigg(\frac{1}{3}x_3 + \frac{1}{6}\bigg)(x_4^2 - x_4) + \bigg(\frac{1}{6}\bigg)(x_5^2 - x_5)~.
\end{align*}
}
\end{exm}

Note that the coefficient for the stable set polynomial
contains every stable set, and further note that every monomial in
every coefficient is also indeed a stable set.

We will now prove that the stability number $\alpha(G)$ is the minimum-degree for any
Nullstellensatz certificate for the non-existence of a stable set
of size greater than $\alpha(G)$~. To prove this, we rely on two
propositions.

\begin{proposition} Given a graph $G$~, let $M = \{d_1,d_2,\ldots,d_{|M|}\}$ be any maximal stable set in $G$~. Let
\begin{align*}
1 &= A\bigg(-(\alpha(G) + r) + \sum_{i=1}^n x_i\bigg) + \sum_{\{i,j\} \in E(G)}Q_{ij}x_ix_j + \sum_{i=1}^nQ_i(x_i^2 - x_i)
\end{align*}
be a Nullstellensatz certificate for the non-existence of a stable set of size $\alpha(G) + r$ (with $r \geq 1$), and let
\begin{align}
1 &= A'\underbrace{
\bigg(-(\alpha(G) + r) + \sum_{i=1}^n x_i \bigg)}_{B} + \underbrace{\sum_{\{i,j\} \in E(G)}Q'_{ij}x_ix_j}_{C} + \underbrace{\sum_{i=1}^nQ'_i(x_i^2 - x_i)}_{D} \label{ss_prop_1_red_cert}
\end{align}
be the reduced certificate via Lemma \ref{reduced}. Then, for $i \in
\{1,\ldots,|M|\}$~, the  linear term $x_{d_i}$ appears in $A'$ with a
non-zero coefficient. \label{ss_prop_1_x_d_i}
\end{proposition}

\begin{proof}
 Our proof is by contradiction. Assume that
$x_{d_i}$ does not appear in $A'$ with a non-zero coefficient. By
inspection of Eq. \ref{ss_prop_1_red_cert}, we see that $A'$ must contain
the constant term $a_0 = -(\alpha(G) + r)$~. Therefore, the term
$a_0x_{d_i}$ appears in $A'B$~. However, $a_0x_{d_i}$ does not cancel
within $A'B$ since $x_{d_i}$ does not appear in $A'$ by assumption.
Therefore, $a_0x_{d_i}$ must cancel with a term elsewhere in the
certificate. Specifically, since $C$ only contains terms multiplying
edge monomials, $a_0x_{d_i}$ must cancel with a term in $D$~. But the
linear term $a_0x_{d_i}$ is generated in $D$ in only one way: $a_0$
must multiply $(x^2_{d_i} - x_{d_i})$ Therefore,
\begin{align*}
Q'_{d_i} &= \text{other terms} + a_0~.
\end{align*}
But when $a_0$ multiplies $(x_{d_i}^2 - x_{d_i})$~, this not only
generates $-a_0x_{d_i}$~,  which neatly cancels its counterpart in
$A'B$~, but it also generates the cross-term $a_0x_{d_i}^2$~, which
must cancel elsewhere in the certificate. However, $x_{d_i}^2$ does
not cancel with a term $C$~, since $C$ contains only terms
multiplying edge monomials, and $x_{d_i}^2$ does not cancel with a
term in $A'B$~, since $A'$ contains only terms corresponding to
square-free stable sets and also because $A'$ does not contain
$x_{d_i}$ by assumption. Therefore,  $a_0x^2_{d_i}$ must cancel
elsewhere in $D$~. There is only one way to generate a second
$x_{d_i}^2$ term in $D$: $a_0x_{d_i}$ must multiply $(x^2_{d_i} -
x_{d_i})$~. Therefore,
\begin{align*}
Q'_{d_i} &= \text{other terms} + a_0x_{d_i} + a_0~.
\end{align*}
Now, we assume
\begin{align*}
Q'_{d_i} &= \text{other terms} + a_0x^k_{d_i} + a_0x^{k-1}_{d_i} \cdots + a_0x^2_{d_i} + a_0x_{d_i} + a_0~.
\end{align*}
When $a_0x^k_{d_i}$ multiplies $(x_{d_i}^2 - x_{d_i})$~, this
generates a cross term  of the form $a_0x^{k+2}_{d_i}$~. This term
must cancel elsewhere in the certificate. As before,
$a_0x^{k+2}_{d_i}$ does not cancel with a term in $C$~, since $C$
contains only terms multiplying edge monomials, and
$a_0x^{k+2}_{d_i}$ does not cancel with a term in $A'B$~, since $A'$
contains only terms corresponding to square-free stable sets.
Therefore, $a_0x^{k+2}_{d_i}$ must cancel elsewhere in $D$~. But as
before, there is only one way to generate a second
$a_0x^{k+2}_{d_i}$ term in $D$: $a_0x^{k+1}_{d_i}$ must multiply
$(x^2_{d_i} - x_{d_i})$~. Therefore,
\begin{align*}
Q'_{d_i} &= \text{other terms} + a_0x^{k+1}_{d_i} + a_0x^k_{d_i} + \cdots + a_0x^2_{d_i} + a_0x_{d_i} + a_0.
\end{align*}
To summarize, we have inductively shown that in order to cancel
lower-order terms,  we are forced to generate terms of higher and
higher degree. In other words, $Q_{d_i}$ contains an
\textit{infinite chain of monomials increasing in degree}. Since
$\deg(Q'_{d_i})$ is finite, this is clearly a contradiction.

Therefore, $x_{d_i}$ must appear in $A'$ with a non-zero coefficient. \end{proof}

\begin{proposition} Given a graph $G$~, let $M = \{d_1,d_2,\ldots,d_{|M|}\}$ be any
maximal stable set in $G$~, and let $\{c_1,c_2,\ldots,c_{k+1}\}$ be any $(k+1)$-subset of $M$ with $k < |M|$~. Let
\begin{align*}
1 &= A\bigg(-(\alpha(G) + r) + \sum_{i=1}^n x_i\bigg) + \sum_{\{i,j\} \in E(G)}Q_{ij}x_ix_j + \sum_{i=1}^nQ_i(x_i^2 - x_i)
\end{align*}
be a Nullstellensatz certificate for the non-existence of a stable set of size $\alpha(G) + r$ (with $r \geq 1$), and let
\begin{align*}
1 &= A'\underbrace{
\bigg(-(\alpha(G) + r) + \sum_{i=1}^n x_i \bigg)}_{B} + \underbrace{\sum_{\{i,j\} \in E(G)}Q'_{ij}x_ix_j}_{C} + \underbrace{\sum_{i=1}^nQ'_i(x_i^2 - x_i)}_{D}
\end{align*}
be the reduced certificate via Lemma \ref{reduced}. If
$x_{c_1}x_{c_2}\cdots x_{c_k}$  appears in $A'$ with a non-zero
coefficient, then $x_{c_1}x_{c_2}\cdots x_{c_k}x_{c_{k+1}}$ also
appears in $A'$ with a non-zero coefficient.
\label{ss_prop_2_subset}
\end{proposition}

\begin{proof}
 Our proof is by contradiction. Assume
$x_{c_1}x_{c_2}\cdots x_{c_k}$  with $k < |M|$ appears in $A'$ with
a non-zero coefficient, but $x_{c_1}x_{c_2}\cdots
x_{c_{k}}x_{c_{k+1}}$ does not. Since $x_{c_1}x_{c_2}\cdots x_{c_k}$
appears in $A'$~, $x_{c_1}x_{c_2}\cdots x_{c_{k}}x_{c_{k+1}}$ clearly
appears in $A'B$ and must cancel elsewhere in the certificate.
However, $x_{c_1}x_{c_2}\cdots x_{c_{k}}x_{c_{k+1}}$ does not cancel
with a term in $C$~, since $C$ contains only terms multiplying edge
monomials and $\{c_1,c_2,\ldots,c_{k+1}\}$ is a stable set.
Furthermore, $x_{c_1}x_{c_2}\cdots x_{c_{k}}x_{c_{k+1}}$ does not
cancel with a term in $A'B$~, since $A'$ does not contain
$x_{c_1}x_{c_2}\cdots x_{c_{k}}x_{c_{k+1}}$ by assumption.
Therefore, $x_{c_1}x_{c_2}\cdots x_{c_{k}}x_{c_{k+1}}$ must cancel
with a term in $D$~, and for at least one $i$~, $x_{c_1}x_{c_2}\cdots
\widehat{x_{c_i}}\cdots x_{c_{k}}x_{c_{k+1}}$ appears in $Q_{c_i}$
with a non-zero coefficient.

When $x_{c_1}\cdots \widehat{x_{c_i}}\cdots x_{c_{k}}x_{c_{k+1}}$
multiplies $(x^2_{c_i} - x_{c_i})$~,  this generates a cross term of
the form $x_{c_1}\cdots x^2_{c_i}\cdots x_{c_{k}}x_{c_{k+1}}$ which
must cancel elsewhere in the certificate. Let $m_1 =
x_{c_1}x_{c_2}\cdots x_{c_{k}}x_{c_{k+1}}$ and $m_2 =
x_{c_1}x_{c_2}\cdots x^2_{c_i}\cdots x_{c_{k}}x_{c_{k+1}}$~. Note
that $\deg(m_2) = \deg(m_1) + 1$~. As before, $m_2$ does not cancel
with a term in $C$~, since $C$ contains only terms multiplying edge
monomials, and $m_2$ does not cancel with a term in $A'B$~, since
$A'$ contains only terms corresponding to square-free stable sets.
Therefore, $m_2$ must cancel elsewhere in $D$~.

In order to cancel $m_2$ in $D$~, for some $c_j$~, we must subtract
one from the $c_j$-th exponent,  and then multiply this monomial by
$(x_{c_j}^2 - x_{c_j})$~. However, this generates a cross-term $m_{3}
= x_{c_1}x_{c_2}\cdots x^2_{c_i}\cdots x^2_{c_j}\cdots
x_{c_{k}}x_{c_{k+1}}$ where $\deg(m_3) = \deg(m_2) + 1$~. Note that
in the case when $j = i$~, $m_3 = x_{c_1}x_{c_2}\cdots
x^3_{c_i}\cdots x_{c_{k}}x_{c_{k+1}}$~, but $\deg(m_3)$ still is
equal to $\deg(m_2) + 1$~.

Inductively, consider the $n$-th element in this chain, and assume it
appears with a non-zero coefficient in some $Q_{c_j}$~. Let
\begin{align*}
m_n &= x_{c_1}^{\alpha_1}x_{c_2}^{\alpha_2} \cdots x_{c_{k+1}}^{\alpha_{k+1}}~,
\end{align*}
where $\alpha_{i} \geq 1$ for $i \in \{1,\ldots, k+1\}$~. When $m_n$
multiplies $(x_{c_j}^2 - x_{c_j})$~, this generates the cross-term
$m_{n}x_{c_j}^2$~. This term must cancel elsewhere in the
certificate. As before, $m_{n}x_{c_j}^2$ does not cancel with a term
$C$~, since $C$ contains only terms multiplying edge monomials, and
$m_{n}x_{c_j}^2$ does not cancel with a term in $A'B$~, since $A'$
contains only terms corresponding to square-free stable sets and
$x_{c_1}x_{c_2}\cdots x_{c_{k}}x_{c_{k+1}}$ does not appear in $A'$ by
assumption. Therefore, $m_{n}x_{c_j}^2$ must cancel elsewhere in $D$~.

In order to cancel $m_{n}x_{c_j}^2$ in $D$~, note that
\begin{align*}
m_n{c_j}^2 &= x_{c_1}^{\alpha_1}x_{c_2}^{\alpha_2} \cdots x_{c_j}^{\alpha_j + 2} \cdots x_{c_{k+1}}^{\alpha_{k+1}}~,
\end{align*}
and for some $l$~, let
\begin{align*}
m_{n+1} &= x_{c_1}^{\alpha_1}x_{c_2}^{\alpha_2} \cdots x_{c_j}^{\alpha_j + 2} \cdots x_{c_l}^{\alpha_l - 1} \cdots x_{c_{k+1}}^{\alpha_{k+1}}.
\end{align*}
Note that $\deg(m_{n+1}) = \deg(m_n) + 1$~. Therefore, in order to
cancel $m_{n}x_{c_j}^2$~, $m_{n+1}$ multiplies $(x_{c_l}^2 - x_{c_l})$~,
which generates a new term of higher degree: $m_{n+1}x^2_{c_l}$~.

To summarize, we have inductively shown that in order to cancel
lower-order terms, we are forced to generate terms of higher and
higher degree. In other words, $m_1,m_2,\ldots$ form an
\textit{infinite chain of monomials increasing in degree}. Since
$\deg(Q'_{i})$ is finite, this is clearly a contradiction.

Therefore, $x_{c_1}x_{c_2}\cdots x_{c_{k}}x_{c_{k+1}}$ must appear in $A'$ with a non-zero coefficient. \end{proof}

Using Propositions \ref{ss_prop_1_x_d_i} and \ref{ss_prop_2_subset}, we can now prove the main theorem of this section.

\begin{theorem} Given a graph $G$~, any Nullstellensatz certificate for the non-existence of a stable set
of size greater than $\alpha(G)$  has degree at least $\alpha(G)$~. \label{thm_ss_cert_min}
\end{theorem}

\begin{proof}
Our proof is by contradiction. Let
\begin{align*}
1 &= A\bigg(-(\alpha(G) + r) + \sum_{i=1}^n x_i\bigg) + \sum_{\{i,j\} \in E(G)}Q_{ij}x_ix_j + \sum_{i=1}^nQ_i(x_i^2 - x_i)
\end{align*}
be any Nullstellensatz certificate for the non-existence of a stable set of size $\alpha(G) + r$, with $r \geq 1$~,
such that $\deg(A),\deg(Q_i),\deg(Q_{ij}) < \alpha(G)$~, and let
\begin{align}
1 &= A'\underbrace{
\bigg(-(\alpha(G) + r) + \sum_{i=1}^n x_i \bigg)}_{B} + \underbrace{\sum_{\{i,j\} \in E(G)}Q'_{ij}x_ix_j}_{C} + \underbrace{\sum_{i=1}^nQ'_i(x_i^2 - x_i)}_{D} \label{ss_min_red_cert}
\end{align}
be the reduced certificate via Lemma \ref{reduced}. The proof of Lemma \ref{reduced}
implies $\deg(A') \leq \deg(A) < \alpha(G)$~. Let $M =
\{d_1,d_2,\ldots,d_{\alpha(G)}\}$ be any maximum stable set in
$G$~. Via Proposition \ref{ss_prop_1_x_d_i}, we know that $x_{d_1}$ appears
in $A'$ with a non-zero coefficient, which implies (via Proposition
\ref{ss_prop_2_subset}) that $x_{d_1}x_{d_2}$ appears in $A'$ with a
non-zero coefficient, which implies that $x_{d_1}x_{d_2}x_{d_3}$
appears in $A'$ and so on. In particular, $x_{d_1}x_{d_2}\cdots
x_{d_{\alpha(G)}}$ appears in $A'$~. This contradicts our assumption
that $\deg(A') < \alpha(G)$~. Therefore, there can be no
Nullstellensatz certificate with $\deg(A) < \alpha(G)$~, and the degree
of \textit{any} Nullstellensatz certificate is at least
$\alpha(G)$~. \end{proof}

Propositions \ref{ss_prop_1_x_d_i} and \ref{ss_prop_2_subset} also give rise to the following corollary.

\begin{corollary} Given a graph $G$~, any Nullstellensatz certificate for the
non-existence of a
stable set of size greater than $\alpha(G)$ contains at least one monomial for every stable set in $G$~.
\label{cor_terms2is}
\end{corollary}

\begin{proof}
 Given any Nullstellensatz certificate, we can create the reduced certificate via Lemma \ref{reduced}.
The proof of the Lemma \ref{reduced} implies that the number of terms in $A$ is equal to the number of terms in $A'$~.
Via Propositions \ref{ss_prop_1_x_d_i} and \ref{ss_prop_2_subset}, $A'$ contains one monomial for every stable set in $G$~.
Therefore, $A$ also contains one monomial for every stable set in $G$~.
\end{proof}

This brings us to the last theorem of this section.

\begin{theorem} Given a graph $G$~, a minimum-degree Nullstellensatz
certificate for the non-existence of a stable set of size greater than $\alpha(G)$
has degree equal to $\alpha(G)$ and contains at least one term for
every stable set in $G$~.
\end{theorem}

\begin{proof} This theorem follows directly from Theorems \ref{thm_ss_cert_exists},
\ref{thm_ss_cert_min}, and Corollary \ref{cor_terms2is}.
\end{proof}

Finally, our results establish new lower bounds for the degree and number of terms
of Nullstellensatz certificates. In
earlier work, researchers in logic and complexity showed both
logarithmic and linear growth of degree of the Nullstellensatz over
finite fields or for special instances, e.g. Nullstellensatz related
to the pigeonhole principle (see \cite{bus}, \cite{impagliazzo} and
references therein). Our main complexity result below settles a question of Lov\'asz \cite{lovasz1}:

\begin{corollary} Given any infinite family of graphs $G_n$, on $n$ vertices,
the degree of a minimum-degree
Nullstellensatz certificate for the non-existence of a stable set of size
greater than $\alpha(G)$ grows as $\Omega(n)$~.  Moreover,
there are  graphs for which the degree of the Nullstellensatz certificate grows
linearly in $n$ and, at the same time, the number of terms in the
coefficient polynomials of the Nullstellensatz certificate is exponential in
$n$~.
\end{corollary}

\begin{proof}
The stability number of a graph $G$ with $n$ nodes
and $m$ vertices grows linearly (\cite{harant}) since
\begin{align*}
\alpha(G) &\geq \frac{1}{2}\bigg((2m + n + 1) - \sqrt{(2m + n + 1)^2 - 4 n^2} \bigg).
\end{align*}
Finally, it is enough to remark that
there exist families of graphs with linear growth in the minimum degree
of their Nullstellensatz certificates, but exponential growth in their numbers of
terms. The disjoint union of $n/3$ triangles has exactly $3^{n/3}$
maximal stable sets. Therefore, its Nullstellensatz certificate's minimum degree
grows as $O(n/3)$~, but its number of terms grows as
$3^{n/3}$ (see \cite{griggs} and references therein). \end{proof}

\subsection{The Nullstellensatz and 3-colorability} \label{HilbertNcolor}

 In this section, we investigate the degree growth of Nullstellensatz
 certificates for the non-$3$-colorability for certain
 graphs. Curiously, every non-3-colorable graph that we have
 investigated so far has a minimum-degree Nullstellensatz certificate
 of degree four.  Next, we prove that four is indeed a lower bound on
 the degree of such certificates.

 \begin{theorem} Every Nullstellensatz certificate of a non-3-colorable graph has degree \emph{at least} four.
 \end{theorem}

\begin{proof}
 Our proof is by contradiction. Suppose there
 exists a Nullstellensatz certificate of degree three or less.
 Such a certificate has the following form
 \begin{align}
 1 &= \sum_{i=1}^nP_{\{i\}}(x_i^3 - 1) + \sum_{\{i,j\} \in E}P_{\{ij\}}(x_i^2 + x_ix_j + x_j^2)~, \label{min_proof:cert_deg_3}
 \end{align}
 where $P_{\{i\}}$ and $P_{\{ij\}}$ represent general
 polynomials of degree less than or equal to three. To be precise,
 \begin{align*}
 P_{\{i\}} = & \sum_{s=1}^na_{\{i\}s}x_s^3 + \sum_{s=1}^n \sum_{\stackrel{t=1}{t\neq s}}^nb_{\{i\}st}x_s^2x_t \\
 & + \sum_{s=1}^n \sum_{t=s+1}^n \sum_{u=t+1}^n c_{\{i\}stu}x_sx_tx_u + \sum_{s=1}^n \sum_{t=1}^nd_{\{i\}st}x_sx_t + \sum_{s=1}^ne_{\{i\}s}x_s + f_{\{i\}}
 \end{align*}
 and
 \begin{align*}
 P_{\{ij\}} =& \sum_{s=1}^na_{\{ij\}s}x_s^3 + \sum_{s=1}^n \sum_{\stackrel{t=1}{t\neq s}}^nb_{\{ij\}st}x_s^2x_t \\
 & + \sum_{s=1}^n \sum_{t=s+1}^n \sum_{u=t+1}^n c_{\{ij\}stu}x_sx_tx_u + \sum_{s=1}^n \sum_{t=1}^nd_{\{ij\}st}x_sx_t + \sum_{s=1}^ne_{\{ij\}s}x_s + f_{\{ij\}}~.
 \end{align*}

 Since we work with undirected graphs,
 note that $a_{\{ij\}s} =a_{\{ji\}s}$~, and this fact applies to all coefficients $a$ through $f$~.
 Note also that when $\{i,j\}$ is not an edge of the
 graph, $P_{ij} =0$ and thus
 $a_{\{ij\}s} = 0$~. Again, this fact holds for  all
 coefficients $a$ through $f$~.

 When $P_{\{i\}}$ multiplies $(x_i^3 - 1)$~, this generates cross-terms
 of the form $P_{\{i\}}x_i^3$ and $-P_{\{i\}}$~. In particular, this
 generates monomials of degree six or less. Notice that $P_{\{ij\}}(x_i^2 +
 x_ix_j + x_j^2)$ does \textit{not} generate monomials of degree six,
 only monomials of degree five or less. We begin the process of
 deriving a contradiction from Eq. \ref{min_proof:cert_deg_3} by
 considering all monomials of the form $x_s^3x_i^3$ that appear in the
 expanded Nullstellensatz certificate. These monomials are formed in
 only \textit{two} ways: Either (1) $x_s^3(x_i^3 - 1)$~, or (2)
 $x_i^3(x_s^3 - 1)$~. Therefore, the $n^2$ equations for $x_s^3x_i^3$
 are as follows:

 \begin{align*}
 0      &=      & a_{\{1\}1}~,                      & \qquad(\text{coefficient for $x_1^3x_1^3 = x_1^6$})      \tag{$I.1$}\\
 0      &=      & a_{\{1\}2} + a_{\{2\}1}~,         & \qquad(\text{coefficient for $x_1^3x_2^3$})              \tag{$I.2$}\\
 \vdots &       & \vdots \qquad\qquad                     & \qquad\qquad\qquad\vdots                                             \\
 0      &=      & a_{\{n-1\}n} + a_{\{n\}(n-1)}~,   & \qquad(\text{coefficient for $x_{(n-1)}^3x_{n}^3$})      \tag{$I.n^2 - 1$}\\
 0      &=      & a_{\{n\}n}~.                      & \qquad(\text{coefficient for $x_{n}^3x_{n}^3$} = x_n^6)  \tag{$I.n^2$}
 \end{align*}

 Summing equations $I.1$ through $I.n^2$~, we get
 \begin{align}
 0 &= \sum_{i = 1}^n\sum_{s = 1}^na_{\{i\}s}~. \label{min_proof:a_i}
 \end{align}
 Let us now consider monomials of the form $x_s^2x_tx_i^3$ (with $s
 \neq t$). These monomials are formed in only \textit{one} way: by
 multiplying $b_{\{i\}st}x_s^2x_t$ by $x_i^3$~. Therefore, since the
 coefficient for $x_s^2x_tx_i^3$ must simplify to zero in the expanded
 Nullstellensatz certificate, $b_{\{i\}st} = 0$ for all
 $b_{\{i\}}$~. When we consider monomials of the form $x_sx_tx_ux_i^3$
 (with $s < t < u)$~, we see that $c_{\{i\}stu} = 0$ for all
 $c_{\{i\}}$~, for the same reasons as above.

 As we continue toward our contradiction, we now consider monomials
 of degree three in the expanded Nullstellensatz certificate. In
 particular, we consider the coefficient for $x_s^3$~. The monomial
 $x_s^3$ is generated in three ways: (1) $f_{\{s\}}(x_s^3 - 1)$~, (2)
 $a_{\{i\}s}x_s^3(x_i^3 - 1)$ (from the vertex polynomials), and (3)
 $e_{\{st\}s}x_s(x_s^2 +x_sx_t + x_t^2)$ (from the edge
 polynomials). The equations for $x_1^3,\ldots,x_n^3$ are as follows:

  \begin{align*}
 0      &=& f_{\{1\}}-\sum_{i=1}^na_{\{i\}1} + \sum_{t \in \text{Adj}(1)}e_{\{1t\}1}~, & \qquad(\text{coefficient for $x_1^3$})  \tag{$II.1$}\\
 0      &=& f_{\{2\}}-\sum_{i=1}^na_{\{i\}2} + \sum_{t \in \text{Adj}(2)}e_{\{2t\}2}~, & \qquad(\text{coefficient for $x_2^3$}) \tag{$II.2$}\\
 \vdots &       & \vdots \qquad\qquad\qquad                     & \qquad\qquad\vdots                                             \\
 0      &=& f_{\{n\}}-\sum_{i=1}^na_{\{i\}n} + \sum_{t \in \text{Adj}(n)}e_{\{nt\}n}~. & \qquad(\text{coefficient for $x_n^3$}) \tag{$II.n$}
 \end{align*}

 Summing equations $II.1$ through $II.n$~, we get
 \begin{align}
 0 &= \sum_{i = 1}^nf_{\{i\}} - \bigg(\sum_{i = 1}^n\sum_{s = 1}^na_{\{i\}s} \bigg) + \sum_{s=1}^n\sum_{t \in \text{Adj}(s)}^ne_{\{st\}s}~.\label{min_proof:x_i^3}
 \end{align}
 Since the degree three or less Nullstellensatz certificate
 (Eq. \ref{min_proof:cert_deg_3}) is identically one, the constant terms must sum
 to one. Therefore, we know $\sum_{i =
   1}^nf_{\{i\}}~=~-1$~. Furthermore, recall that $e_{\{st\}s} = 0$ if the
 undirected edge $\{s,t\}$ does not exist in the graph. Therefore,
 applying Eq. \ref{min_proof:a_i} to Eq. \ref{min_proof:x_i^3}, we have
 the following equation
 \begin{align}
 1 &= \sum_{\stackrel{s,t=1,}{s \neq t}}^ne_{\{st\}s}~.\label{min_proof:const}
 \end{align}
 To give a preview of our overall proof strategy, the equations to come
 will ultimately show that the right-hand side of Eq. \ref{min_proof:const} also equals zero,
 which is a contradiction.

 Now we will consider the monomial $x_s^2x_t$ (with $s \neq t$). We
 recall that $b_{\{i\}st} = 0$ for all $b_{\{i\}}$ (where $b_{\{i\}st}$
 is the coefficient for $x_s^2x_t$ in the $i$-th vertex
 polynomial). Therefore, we do \textit{not} need to consider
 $b_{\{i\}st}$ in the equation for the coefficient of monomial
 $x_s^2x_t$~. In other words, we only need to consider the edge
 polynomials, which can generate this monomial in two ways: (1)
 $e_{\{st\}s}x_s\cdot x_sx_t$~, and (2) $e_{\{si\}t}x_t \cdot x_s^2$~. The
 $N=2\binom{n}{2}$ equations for these coefficients are:

   \begin{align*}
 0 &=& e_{\{12\}1} + \sum_{i \in \text{Adj}(1)}e_{\{1i\}2}~,&  \qquad(\text{coefficient for $x_1^2x_2$})  \tag{$III.1$}\\
 0 &=& e_{\{13\}1} + \sum_{i \in \text{Adj}(1)}e_{\{1i\}3}~,&  \qquad(\text{coefficient for $x_1^2x_3$})  \tag{$III.2$}\\
 \vdots &       & \vdots \qquad\qquad                     & \qquad\qquad\qquad\vdots                                             \\
 0 &=& e_{\{n(n-1)\}n} + \sum_{i \in \text{Adj}(n)}e_{\{ni\}(n-1)}~.&  \qquad(\text{coefficient for $x_{n}^2x_{n-1}$}) & \tag{$III.N$}
  \end{align*}

 When we sum equations $III.1$ to $III.N$~, we obtain
 \begin{align}
 \sum_{s=1}^n \sum_{\stackrel{t=1,}{t \neq s}}^n e_{\{st\}s} + \underbrace{\bigg(\sum_{s=1}^n \sum_{t \in Adj(s)}e_{\{st\}t}\bigg)}_{\text{partial sum A}}  + \underbrace{\bigg(\sum_{s=1}^n \sum_{t \in Adj(s)} \sum_{\stackrel{u=1,}{u \neq s,t}}^ne_{\{st\}u}\bigg)}_{\text{partial sum B}}&= 0~. \label{min_proof:x_s^2x_t_a}
 \end{align}
 However, recall that $e_{\{st\}u} = 0$ when $\{s,t\}$
 does not exist in the graph, and also that $e_{\{st\}t} = e_{\{ts\}t}$~.
 Thus, we can rewrite partial sum A from Eq. \ref{min_proof:x_s^2x_t_a} as
 \begin{align*}
 \sum_{s=1}^n \sum_{t \in Adj(s)}e_{\{st\}t} =
 \sum_{s=1}^n\sum_{\stackrel{t=1,}{t \neq s}}^n e_{\{st\}t} = \sum_{s=1}^n\sum_{\stackrel{t=1,}{t \neq s}}^n e_{\{ts\}t}~.
 \end{align*}
Substituting the above into Eq. \ref{min_proof:x_s^2x_t_a} yields
 \begin{align}
 2\sum_{\stackrel{s,t=1,}{s \neq t}}^ne_{\{st\}s} + \underbrace{\bigg(\sum_{s=1}^n \sum_{t \in Adj(s)} \sum_{\stackrel{u=1,}{u \neq s,t}}^ne_{\{st\}u}\bigg)}_{\text{partial sum B}}&= 0~. \label{min_proof:x_s^2x_t}
 \end{align}
 Finally, we consider the monomial $x_sx_tx_u$ (with $s < t < u$).
 We have already argued that $c_{\{i\}stu} = 0$ for all $c_{\{i\}}$
 (where $c_{\{i\}stu}$ is the coefficient for $x_sx_tx_u$ in the $i$-th vertex polynomial).
 Therefore, as before, we need only consider the edge polynomials, which can generate this monomial
 in three ways: (1) $e_{\{st\}u}x_u \cdot x_sx_t$~, (2) $e_{\{su\}t}x_t \cdot x_sx_u$~, and
 (3) $e_{\{tu\}s}x_s\cdot x_tx_u$~. As before, these coefficients must cancel in the expanded certificate,
 which yields the following $\binom{n}{3}$ equations:

\begin{align*}
 0 &=& e_{\{12\}3} + e_{\{13\}2} + e_{\{23\}1}~, &  \qquad(\text{coefficient for $x_1x_2x_3$}) & \tag{$IV.1$}\\
 0 &=& e_{\{12\}4} + e_{\{14\}2} + e_{\{24\}1}~, &  \qquad(\text{coefficient for $x_1x_2x_4$}) &\tag{$IV.2$}\\
 \vdots &       & \vdots \qquad\qquad                     & \qquad\qquad\qquad\vdots                                             \\
 0 &=& e_{\{(n-2)(n-1)\}n} + e_{\{(n-2)n\}(n-1)} + e_{\{(n-1)n\}(n-2)}~.&  \qquad(\text{coefficient for $x_{n-2}x_{n-1}x_n$}) \tag{$IV.M$}\\
 \end{align*}

Summing equations $IV.1$ through $IV.M$~, we obtain
 \begin{align}
 \sum_{s=1}^{n-2}\sum_{t=s+1}^{n-1}\sum_{u=t+1}^n\bigg(e_{\{st\}u} + e_{\{su\}t} + e_{\{tu\}s}\bigg)  &= 0~. \label{min_proof:x_sx_ux_t}
 \end{align}

Now we come to the critical argument of the proof. 
We claim that the following equation holds:
\begin{align}
\bigg(\sum_{s=1}^n \sum_{t \in Adj(s)} \sum_{\stackrel{u=1,}{u \neq
    s,t}}^ne_{\{st\}u}\bigg) &=
2\bigg(\sum_{s=1}^{n-2}\sum_{t=s+1}^{n-1}\sum_{u=t+1}^n\bigg(e_{\{st\}u}
+ e_{\{su\}t} +
e_{\{tu\}s}\bigg)\bigg)~. \label{min_cert:partial_sum_x_s^2x_t}
\end{align}
Notice that the left-hand and right-hand sides of this equation consist only
of  coefficients $e_{\{st\}u}$ with $s,t,u$ distinct. Consider any
such coefficient $e_{\{st\}u}$~. Notice that $e_{\{st\}u}$ appears
exactly \textit{once} on the right side of the equation.
Furthermore, either $e_{\{st\}u}$ appears exactly \textit{twice} on
the left side of this equation (since $s \in \text{Adj}(t)$ implies
$t \in \text{Adj}(s)$), or $e_{\{st\}u} = 0$ (since the edge
$\{s,t\}$ does not exist in the graph). Therefore, Eq.
\ref{min_cert:partial_sum_x_s^2x_t} is proven. Applying this result
(and Eq. \ref{min_proof:x_sx_ux_t}) to Eq. \ref{min_proof:x_s^2x_t}
gives us the following:

\begin{align}
\sum_{\stackrel{1 \leq s,t \leq n}{s \neq t}}e_{\{st\}s}&= 0~. \label{min_proof:final_x_s^2x_t}
\end{align}
But Eq. \ref{min_proof:final_x_s^2x_t} contradicts Eq.
\ref{min_proof:const} ($1 = 0$),  thus there can be no certificate
of degree less than four. \end{proof}

It is important to note that when we try to construct certificates of
degree four or greater, the equations for the degree-6 monomials
become considerably more complicated. In this case,
the edge polynomials \textit{do} contribute monomials of degree six,
which causes the above argument to break.

We conclude this subsection with a result that allows us to bound
the degree of a minimum-degree  Nullstellensatz certificate of a
particular graph, if that graph can be ``reduced'' to another graph
whose minimum-degree Nullstellensatz certificate is known.

\begin{lemma} \label{lem_node_contraction}
\begin{enumerate}
\renewcommand{\theenumi}{\arabic{enumi}.}
\item  If $H$ is a subgraph of $G$~, and $H$ has a minimum-degree non-3-colorability Nullstellensatz
certificate of degree $k$~, then
$G$ also has a minimum-degree non-3-colorability Nullstellensatz certificate of degree $k$~.
\item Suppose that a non-3-colorable
graph $G$ can be transformed
to a non-3-colorable graph $H$ via a sequence of identifications of non-adjacent
nodes of $G$~. If a minimum-degree non-3-colorability
Nullstellensatz certificate for $H$ has degree $k$~, then a minimum-degree non-3-colorable
Nullstellensatz certificate for $G$ has
degree at least $k$~.
\end{enumerate}
\end{lemma}

\begin{proof}

\noindent \textit{Proof of 1:} Since $H$ is a subgraph of $G$~, then any Nullstellensatz certificate
for non-3-colorability of $H$
is also a  Nullstellensatz certificate for non-3-colorability of $G$~.

\noindent \textit{Proof of 2:}
Suppose that $G$ has a  Nullstellensatz certificate for non-3-colorability of degree less than $k$~.
The certificate has the form $1=\sum a_i v_i +\sum b_{\{ij\}} e_{\{ij\}}$ where $v_i = x_i^3 - 1$~,
$e_{\{ij\}}= x_i^2 + x_ix_j + x_j^2$~, and both
$a_i$ and $b_{\{ij\}}$ denote polynomials of degree less than $k$~. Since the certificate is an identity,
the identity must hold for all values of the
variables. In particular, it must hold for every variable substitution $x_i= x_j$ when the nodes
are non-adjacent. In this case, the variable reassignment (pictorially represented in
Figure \ref{fig:odd_wheel_not_proof}) yields a
Nullstellensatz certificate of degree less than $k$ for the transformed graph $H$~.
Note that the parallel edges that may arise
are irrelevant to our considerations (see such examples in Figure \ref{fig:odd_wheel_not_proof}).
\begin{figure}[h]
\begin{center}
\includegraphics[width=10cm]{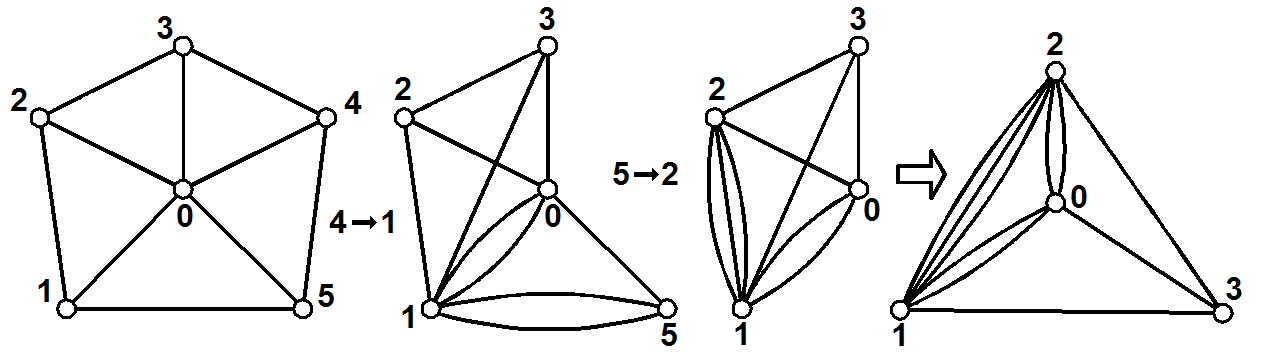}
\end{center}
\caption{Converting $G$ (the 5-odd-wheel) to $H$ (the 3-odd-wheel) via node identifications.}
\label{fig:odd_wheel_not_proof}
\end{figure}
But this is in contradiction with the assumed degree of a minimum-degree certificate for $H$~.
Therefore, any certificate for $G$ must have degree at least $k$~.
\end{proof}

\subsubsection{Cliques, Odd Wheels, and Nullstellensatz Certificates}

\begin{theorem} \label{tcliques}
For $K_n$ with $n \geq 4$~, a minimum-degree  Nullstellensatz certificate for non-3-colorability
has degree exactly four.
\end{theorem}

\begin{proof} It is easy to see that $K_4$ is a subgraph of $K_5$~, which is
a subgraph of $K_6$~, and so on. The decision problem of whether $K_4$ is 3-colorable can be
encoded by the system of equations
\[
\begin{array}{llllll}
&x_1^3 - 1 = 0~, &x_2^3 - 1 =0~, &x_1^2 + x_1x_2 + x_2^2 = 0~, &x_1^2 + x_1x_3 + x_3^2 = 0~, &x_1^2 + x_1x_4 + x_4^2 = 0~,\\
&x_3^3 - 1 = 0~, &x_4^3 - 1 = 0~, & x_2^2 + x_2x_3 + x_3^2 = 0~, &x_2^2 + x_2x_4 +x_4^2 = 0~, &x_3^2 + x_3x_4 + x_4^2 = 0~,
\end{array}
\]
with one equation per vertex and one equation per edge. Using the linear-algebra heuristic described
at the beginning of Section \ref{preliminariesNULL}, we find that $K_4$ has a degree-4 Nullstellensatz certificate
for non-3-colorablility:
\footnotesize
\begin{align*}
1 &= (-x_1^3-1)(x_1^3-1)+\bigg(\frac{4}{9}x_4^4-\frac{5}{9}x_4^3x_2-\frac{2}{9}x_4^3x_3-\frac{4}{9}x_4^3x_1+\frac{2}{9}x_4^2x_2x_1+\frac{2}{9}x_4^2x_3x_1\bigg)(x_4^2+x_2x_4+x_2^2)\\
&\phantom{=} + \bigg(\frac{1}{9}x_4^4+\frac{2}{9}x_4^3x_2-\frac{1}{9}x_4^3x_1-\frac{2}{9}x_4^2x_2x_1 \bigg)(x_2^2+x_3x_2+x_3^2)+ \bigg(\frac{2}{9}x_4^4+\frac{1}{9}x_4^3x_2+\frac{1}{9}x_4^3x_1+\frac{2}{9}x_4^2x_2x_1 \bigg)(x_4^2+x_3x_4+x_3^2) \\
&\phantom{=} + \bigg(-\frac{2}{3}x_4^4+x_4^3x_1-x_4x_1^3+x_1^4\bigg)(x_4^2+x_1x_4+x_1^2)+\frac{1}{3}x_4^3x_2(x_2^2+x_1x_2+x_1^2)+\bigg(-\frac{1}{3}x_4^4-\frac{1}{3}x_4^3x_2\bigg)(x_3^2+x_1x_3+x_1^2).
\end{align*}
\normalsize
Because $K_4$ has a degree-4 Nullstellensatz certificate
as shown above, $K_n$ for $n \geq 4$ also has a degree-4 Nullstellensatz
certificate via Lemma \ref{lem_node_contraction} (1). \end{proof}

The odd-wheels consist of an odd-cycle rim, with a center vertex
connected to all other vertices. It is rather easy to see that no odd-wheel is 3-colorable.
It is natural to ask about the degree of a
minimum-degree  Nullstellensatz certificate for non-3-colorablility.

\begin{theorem} \label{towheels}
Given any odd-wheel, the degree of a minimum-degree
Nullstellensatz certificate for non-3-colorability is four.
\end{theorem}

\begin{proof}
 First, we will prove that there exists a certificate of degree four for the $n$-th odd-wheel. We begin by displaying a degree-4 certificate for the 3-odd-wheel:
\begin{align}
1 &= \bigg(\frac{4}{9}x_1^4-\frac{5}{9}x_1^3x_2-\frac{2}{9}x_1^3x_3-\frac{4}{9}x_1^3x_0+\frac{2}{9}x_1^2x_2x_0+\frac{2}{9}x_1^2x_3x_0\bigg)(x_1^2+x_2x_1+x_2^2)+\nonumber \\
&\hspace{15pt}\bigg(\frac{1}{9}x_1^4+\frac{2}{9}x_1^3x_2-\frac{1}{9}x_1^3x_0-\frac{2}{9}x_1^2x_2x_0\bigg)(x_2^2+x_3x_2+x_3^2)+ \frac{1}{3}x_1^3x_2(x_2^2+x_0x_2+x_0^2) + \nonumber \\
& \hspace{15pt} \bigg(\frac{2}{9}x_1^4+\frac{1}{9}x_1^3x_2+\frac{1}{9}x_1^3x_0+\frac{2}{9}x_1^2x_2x_0\bigg)(x_1^2+x_3x_1+x_3^2) + \frac{1}{3}x_1^4(x_1^2+x_0x_1+x_0^2)+\nonumber \\
& \hspace{15pt}\bigg (-\frac{1}{3}x_1^4-\frac{1}{3}x_1^3x_2\bigg)(x_3^2+x_0x_3+x_0^2) + (-x_1^3-1)(x_1^3-1). \label{eq:cert_3}
\end{align}
For now, we denote the non-3-colorability certificate for the 3-odd-wheel as follows:
\begin{align*}
1 &= \alpha_{\{12\}}e_{\{12\}} + \alpha_{\{23\}}e_{\{23\}} + \alpha_{\{20\}}e_{\{20\}} + \alpha_{\{13\}}e_{\{13\}} +\alpha_{\{10\}}e_{\{10\}} + \alpha_{\{30\}}e_{\{30\}} + \alpha_1v_1,
\end{align*}
where $v_1 = x_1^3 - 1$~, and $e_{\{ij\}} = x_i^2 + x_ix_j + x_j^2$ and $\alpha_1$ and $\alpha_{\{ij\}}$ denote polynomials of degree four in $\mathbb{R}[x_0,x_1,x_2,x_3]$~.
\begin{figure}[h]
\begin{center}
\includegraphics[width=7cm]{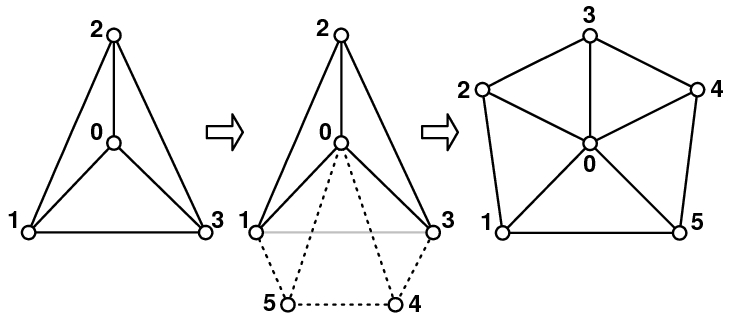}
\end{center}
\caption{Here we show the evolution of the 3-odd-wheel to the 5-odd-wheel. This can be extended from the $n$-th odd wheel to the $(n+2)$-th odd wheel.}
\label{fig:odd_wheel_proof}
\end{figure}
In Figure \ref{fig:odd_wheel_proof}, we can see that the topological
difference between the 3-odd-wheel and the 5-odd-wheel is that the
edge $(1,3)$ is lost, and the vertices $4,5$ and associated edges
$(3,4), (4,5), (5,1), (0,4)$ and $(0,5)$ are gained. We can exhibit an algebraic
relation (or syzygy) as follows:
\begin{align}
\alpha_{\{13\}}e_{\{13\}} &= \alpha_{\{13\}}e_{\{15\}} + \beta_{\{34\}}e_{\{34\}} + \beta_{\{45\}}e_{\{45\}} + \beta_{\{01\}}e_{\{01\}} + \beta_{\{03\}}e_{\{03\}} + \beta_{\{04\}}e_{\{04\}} \nonumber\\
& \hspace{10pt} + \beta_{\{05\}}e_{\{05\}}~, \label{eq:odd_wheel_syzygy}
\end{align}
where $\beta_{\{ij\}} \in \mathbb{R}[x_0,x_1,x_2,x_3,x_4,x_5]$ and $\deg(\beta_{\{ij\}}) \leq 4$~. Note that the coefficients for $e_{\{13\}}$ and $e_{\{15\}}$ are the same. Recall that
\begin{align*}
\alpha_{\{13\}} &= \frac{2}{9}x_1^4+\frac{1}{9}x_1^3x_2+\frac{1}{9}x_1^3x_0+\frac{2}{9}x_1^2x_2x_0~.
\end{align*}
Since $\alpha_{\{13\}}$ does not contain the variable $x_3$~,
$\alpha_{\{13\}} \in \mathbb{R}[x_0,x_1,x_2]$ and \textit{not} in
$\mathbb{R}[x_0,x_1,x_2,x_3]$~. Therefore, because there exists a degree-4
certificate for the 3-odd-wheel, we can simply use the above
syzygy (Eq. \ref{eq:odd_wheel_syzygy}) to substitute for the
$\alpha_{\{13\}}e_{\{13\}}$ term in the degree-4 3-odd-wheel
certificate (Eq. \ref{eq:cert_3}). The resulting polynomial is a
degree-4 certificate for the 5-odd-wheel where the coefficient for $e_{\{15\}}$ in the
5-odd-wheel certificate is exactly the same as the coefficient for $e_{\{13\}}$ in the
3-odd-wheel certificate (both coefficients are $\alpha_{\{13\}}$). Thus, we can again
use the syzygy of Eq. \ref{eq:odd_wheel_syzygy} (with the variable
substitutions of $x_4 \rightarrow x_6$~, $x_5 \rightarrow x_7$~, and $x_3 \rightarrow x_5$), and
substitute for the $\alpha_{\{13\}}e_{\{15\}}$ term in the 5-odd-wheel
certificate to obtain a degree-4 7-odd-wheel certificate. Thus, by
induction, we obtain degree-4 certificates for \textit{all}
odd-wheels. It remains for us to show that such a syzygy exists.

The special syzygy was found via the linear algebra heuristic described at the beginning of
Section \ref{preliminariesNULL} and is listed below. Note that the coefficients for $e_{\{13\}}$
and $e_{\{15\}}$ are indeed identical and are equal to the coefficient for $e_{\{13\}}$ in the
3-odd-wheel certificate we presented earlier (Eq. \ref{eq:cert_3}).
\footnotesize
\begin{align*}
0 &= - \underbrace{\bigg(\frac{2}{9}x_1^4+\frac{1}{9}x_1^3x_2+\frac{1}{9}x_1^3x_0+\frac{2}{9}x_1^2x_2x_0 \bigg)}_{\alpha_{\{13\}}}\underbrace{(x_1^2+x_3x_1+x_3^2)}_{e_{\{13\}}}+ \underbrace{\bigg(\frac{2}{9}x_1^4+\frac{1}{9}x_1^3x_2+\frac{1}{9}x_1^3x_0+\frac{2}{9}x_1^2x_2x_0 \bigg)}_{\alpha_{\{13\}}} \underbrace{(x_1^2+x_5x_1+x_5^2)}_{e_{\{15\}}} \\
& \hspace{10pt} + \bigg(\frac{2}{9}x_1^3x_0+\frac{1}{9}x_1x_2x_0x_5-\frac{1}{9}x_1x_2x_4x_5-\frac{1}{9}x_1x_3x_0^2-\frac{2}{9}x_1x_3x_0x_4-\frac{2}{9}x_2x_0^3-\frac{1}{9}x_2x_0^2x_4+\frac{1}{9}x_4^4 \bigg) \underbrace{(x_3^2+x_3x_4+x_4^2)}_{e_{\{34\}}} \\
 & \hspace{10pt} + \bigg(-\frac{2}{9}x_1^4-\frac{2}{9}x_1^2x_2x_0-\frac{1}{9}x_1^2x_2x_4+\frac{1}{9}x_1^2x_0x_4-\frac{1}{9}x_1x_2x_3x_0+\frac{1}{9}x_1x_2x_3x_4-\frac{1}{9}x_1x_2x_0^2+\frac{1}{9}x_1x_2x_4^2-\frac{2}{9}x_0^4 \\
& \hspace{28pt} + \frac{1}{9}x_0^3x_4-\frac{1}{9}x_4^4+\frac{1}{9}x_4^3x_5-\frac{1}{9}x_4x_5^3\bigg)\underbrace{(x_4^2+x_4x_5 +x_5^2)}_{e_{\{45\}}} \\
 & \hspace{10pt} + \bigg(-\frac{1}{3}x_1x_3x_0^2-\frac{2}{9}x_3x_0x_4^2-\frac{5}{9}x_1x_3^2x_0-\frac{1}{3}x_1^2x_3x_0+\frac{2}{9}x_1^2x_4x_5+\frac{2}{9}x_0^2x_4x_5-\frac{1}{9}x_1x_4x_5^2+\frac{2}{9}x_3^2x_0x_4+\frac{2}{9}x_2x_3x_4^2 \\
 & \hspace{28pt} + \frac{1}{9}x_1^2x_2x_3-\frac{1}{9}x_1^2x_2x_5+\frac{2}{9}x_1^3x_3-\frac{2}{9}x_1^3x_5+\frac{1}{9}x_1^2x_0x_5-\frac{2}{9}x_1^2x_0^2 +\frac{2}{9}x_1^2x_4^2-\frac{4}{9}x_1x_3^2x_4-\frac{2}{3}x_1x_3x_0x_4-\frac{4}{9}x_1x_0x_4x_5 \\
& \hspace{28pt} -\frac{5}{9}x_1x_0^2x_4-\frac{4}{9}x_1x_0x_4^2-\frac{1}{9}x_1x_0x_5^2-\frac{1}{9}x_1x_4^2x_5-\frac{2}{9}x_1x_0^3+\frac{2}{9}x_2x_3^2x_0+\frac{1}{9}x_2x_3^2x_4-\frac{1}{9}x_2x_3x_5^2+\frac{2}{9}x_2x_0x_4^2+\frac{1}{3}x_2x_3x_0x_4 \\
& \hspace{28pt} -\frac{1}{9}x_2x_3x_0x_5+\frac{1}{9}x_2x_4^3-\frac{4}{9}x_3^3x_0-\frac{1}{3}x_3^4-\frac{1}{9}x_3^3x_4+\frac{2}{9}x_3^2x_4^2+\frac{2}{9}x_0^2x_5^2-\frac{1}{9}x_0x_4^3 \bigg)\underbrace{(x_0^2+x_0x_1+x_1^2)}_{e_{\{01\}}} \\
& \hspace{10pt} + \bigg(\frac{2}{9}x_1^4+\frac{1}{9}x_1^3x_2+\frac{4}{9}x_1^3x_0+\frac{4}{9}x_1^3x_4-\frac{1}{9}x_1^2x_2x_4+\frac{1}{3}x_1^2x_3^2+\frac{1}{9}x_1^2x_3x_0+\frac{1}{9}x_1^2x_3x_4+\frac{5}{9}x_1^2x_0^2 +\frac{5}{9}x_1^2x_0x_4+\frac{2}{9}x_1^2x_4^2 \\
& \hspace{28pt} -\frac{2}{9}x_1x_2x_0^2-\frac{1}{9}x_1x_2x_0x_4-\frac{1}{9}x_1x_2x_0x_5+\frac{1}{9}x_1x_2x_4x_5+\frac{1}{3}x_1x_3^2x_0+\frac{2}{9}x_1x_3x_0^2+\frac{1}{3}x_1x_3x_0x_4+\frac{1}{3}x_3^2x_0^2 \\
 & \hspace{28pt} -\frac{1}{9}x_3x_0^3-\frac{1}{9}x_3x_0^2x_4-\frac{2}{9}x_3x_0x_4^2-\frac{2}{9}x_0^4-\frac{2}{9}x_0^3x_4 \bigg)\underbrace{(x_0^2+x_0x_3+x_3^2)}_{e_{\{03\}}} \\
& \hspace{10pt} + \bigg(\frac{1}{9}x_1^3x_5-\frac{2}{9}x_1^2x_2x_3+\frac{1}{9}x_1^2x_2x_5-\frac{4}{9}x_1^2x_3^2-\frac{1}{9}x_1x_2x_3x_4+\frac{1}{9}x_1x_2x_0^2-\frac{1}{9}x_1x_2x_4^2+\frac{1}{9}x_1x_3x_0^2+\frac{2}{9}x_1x_3x_0x_4 \\
& \hspace{28pt} +\frac{1}{3}x_1x_0^3 +\frac{1}{9}x_1x_0^2x_4+\frac{1}{9}x_1x_0^2x_5+\frac{1}{9}x_2x_3x_0x_5+\frac{1}{9}x_2x_3x_5^2+\frac{2}{9}x_3^3x_0+\frac{1}{9}x_3^2x_0x_4-\frac{1}{9}x_3^2x_4^2+\frac{1}{3}x_3x_0^3 \\
& \hspace{28pt}  +\frac{1}{9}x_3x_0x_4^2-\frac{1}{9}x_3x_4^3+\frac{2}{9}x_0^4 \bigg)\underbrace{(x_0^2+x_0x_4+x_4^2)}_{e_{\{04\}}} \\
 & \hspace{10pt} +  \bigg(-\frac{1}{9}x_1^3x_2+\frac{1}{9}x_1^3x_4+\frac{1}{9}x_1^2x_2x_3+\frac{1}{9}x_1^2x_2x_4-\frac{1}{9}x_1^2x_0^2+\frac{2}{9}x_1x_2x_3x_0-\frac{1}{9}x_1x_2x_3x_4+\frac{1}{9}x_1x_2x_0^2-\frac{1}{9}x_1x_2x_4^2 \\
& \hspace{28pt} -\frac{1}{9}x_1x_0^3+\frac{1}{9}x_1x_0^2x_4-\frac{1}{9}x_2x_3x_0x_4-\frac{1}{9}x_2x_3x_4^2-\frac{1}{9}x_0x_4^2x_5-\frac{1}{9}x_0x_4x_5^2+\frac{1}{9}x_4^2x_5^2+\frac{1}{9}x_4x_5^3 \bigg)\underbrace{(x_0^2+x_0x_5+x_5^2)}_{e_{\{05\}}}~.
\end{align*}
\normalsize
\end{proof}

Finally, the reader may easily observe that Theorem \ref{degHN}, Part 2 follows directly
from Theorem \ref{tcliques}, Theorem \ref{towheels}, and Lemma \ref{lem_node_contraction}.

\subsection{Nullstellensatz Certificates for Other Non-3-Colorable Graphs}
With the aid of a computer, we searched many families of
non-3-colorable graphs, hoping to find explicit examples with
growth in the certificate degree. Every graph we have investigated so far
has a Nullstellensatz certificate of degree four. This suggests that
examples with degree growth are rare for graphs with few vertices,
and that many graphs have short proofs of non-3-colorability. In
Figure \ref{fig_uc_jin_grotzch}, we describe the Jin and Gr\"otzch
graphs, and in Figure \ref{fig_flowers}, we describe the ``Flower"
family. Kneser graphs are described in most graph theory books. In
Table \ref{tbl_gc_exp_flwr_jin_kneser}, we present a sampling of the
many graphs we tried during our computational experiments. Note that
we often used our probabilistic linear algebra algorithm, selecting
$p = .4$ as a likely threshold for feasibility.

\begin{table}[h]
\begin{center}
\begin{tabular}{|c|c|c|c|c|c|c|}
\hline
\em Graph & \em vertices & \em edges & \em row & \em col & \em p & \em $\deg$\\
\hline
\text{flower} 8 & 16 & 32 & 51819 & 49516 & .4 & 4 \\
\hline
\text{flower} 10 & 20 & 40 & 178571 & 362705 & 1 & 4 \\
\hline
\text{flower} 11 & 22 & 44 & 278737 & 278844 & .5 & 4 \\
\hline
\text{flower} 13 & 26 & 52 & 629666 & 495051 & .4 & 4 \\
\hline
\text{flower} 14 & 28 & 56 & 923580 & 705536 & .4 & 4 \\
\hline
\text{flower} 16 & 32 & 64 & 1979584 & 1674379 & .4 & 4 \\
\hline
\text{flower} 17 & 34 & 68 & 2719979 & 2246535 & .4 & 4 \\
\hline
\text{flower} 19 & 38 & 76 & 4862753 & 3850300 & .5 & 4 \\
\hline
\text{kneser-}(6,2) & 15 & 45 & 39059 & 68811 & .5 & 4 \\
\hline
\text{kneser-}(7,2) & 21 & 105 & 230861 & 558484 & .5 & 4 \\
\hline
\text{kneser-}(8,2) & 28 & 210 & 1107881 & 3307971 & .5 & 4 \\
\hline
\text{kneser-}(9,2) & 36 & 378 & 1107955 & 3304966 & .5 & 4 \\
\hline
\text{kneser-}(10,2) & 45 & 630 & 15,567,791 & 36,785,283 & .5 & 4 \\
\hline
\text{jin graph} & 12 & 24 & 12168 & 13150 & .4 & 4 \\
\hline
\text{Gr\"otzsch} & 11 & 20 & 7903 & 8109 & .4 & 4 \\
\hline
$G + \{(3,4)\}$ & 12 & 24 & 12,257 & 13,091 & .4 & 4 \\
\hline
$G + \{(7,12)\}$ & 12 & 24 & 12,201 & 13,085 & .4 & 4 \\
\hline
$G + \{(1,8)\}$ & 12 & 24 & 12,180 & 13,124  & .4 & 4 \\
\hline
$G + \{(3,4),(12,7)\}$ & 12 & 25 & 12,286 & 13,804 & .4 & 4 \\
\hline
\end{tabular}
\caption{Experimental investigations for Flowers, Kneser, the Jin graph and the Gr\"otzch graph.
Here $G$ denotes the uniquely-colorable graph displayed in Figure \ref{fig_uc_jin_grotzch}.}
\label{tbl_gc_exp_flwr_jin_kneser}
\end{center}
\end{table}
\begin{figure}[h]
\begin{center}
\includegraphics[width=10cm, trim=0 0 0 0]{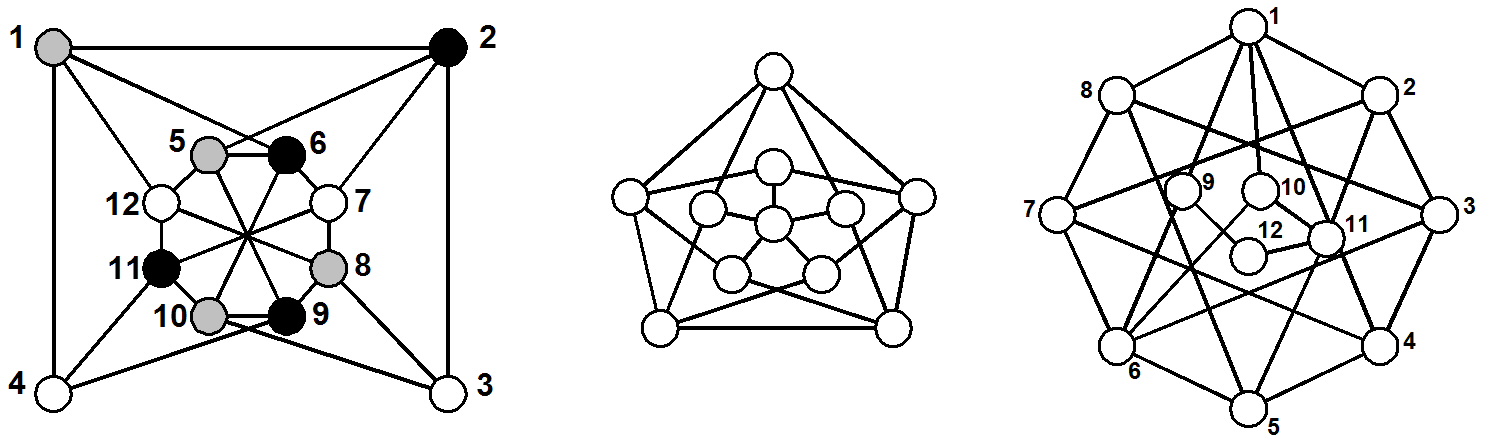}
\end{center}
\caption{These graphs (from left to right) are (1) a uniquely 3-colorable graph, labeled with its unique
3-coloring \cite{chao_chen_uc}, (2) the Gr\"otzsch graph, and (3) the Jin graph.}
\label{fig_uc_jin_grotzch}
\end{figure}
A \textit{uniquely 3-colorable} graph is a graph that can be colored with three colors in only one way,
up to permutation of the color labels. Figure \ref{fig_uc_jin_grotzch} displays a uniquely 3-colorable
triangle-free graph \cite{chao_chen_uc}. Since the graph is uniquely 3-colorable, the addition of a single
edge between two similarly-colored vertices will result in a new non-3-colorable graph.
Table \ref{tbl_gc_exp_flwr_jin_kneser} also details these experiments. Finally, we investigated
all non-3-colorable graphs on six vertices or less: every one has a Nullstellensatz certificate of degree four.
\begin{figure}[h]
\begin{center}
\includegraphics[width=10cm, trim=0 0 0 0]{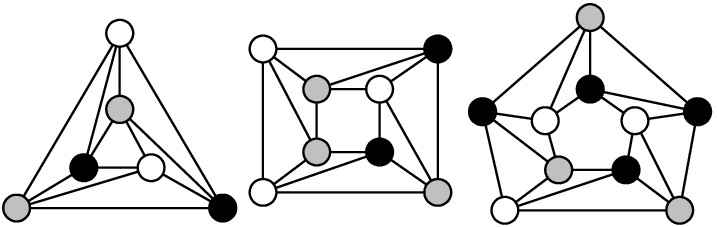}
\caption{$3,4$ and 5 flowers (left to right). Note that the 3-flower is 3-colorable, whereas the 4 and 5
flowers are non-3-colorable. It is easy to see that only flowers that are multiples of 3 are 3-colorable.}
\label{fig_flowers}
\end{center}
\end{figure}


\end{document}